\documentclass[12pt]{amsart}
\usepackage{amssymb,float}
\usepackage[colorlinks=true]{hyperref}
\usepackage[all]{xy}
\usepackage[toc,page,title,titletoc,header]{appendix}
\usepackage{xcolor}
\usepackage{tikz-cd}
\usepackage{graphicx}

\begin{document}
	\pdfoutput=1
	\theoremstyle{plain}
	\newtheorem{thm}{Theorem}[section]
	\newtheorem*{thm1}{Theorem 1}
	\newtheorem*{thm1.1}{Theorem 1.1}
	\newtheorem*{thmM}{Main Theorem}
	\newtheorem*{thmA}{Theorem A}
	\newtheorem*{thm2}{Theorem 2}
	\newtheorem{lemma}[thm]{Lemma}
	\newtheorem{lem}[thm]{Lemma}
	\newtheorem{cor}[thm]{Corollary}
	\newtheorem{pro}[thm]{Proposition}
	\newtheorem{propose}[thm]{Proposition}
	\newtheorem{variant}[thm]{Variant}
	\theoremstyle{definition}
	\newtheorem{notations}[thm]{Notations}
	\newtheorem{rem}[thm]{Remark}
	\newtheorem{rmk}[thm]{Remark}
	\newtheorem{rmks}[thm]{Remarks}
	\newtheorem{defi}[thm]{Definition}
	\newtheorem{exe}[thm]{Example}
	\newtheorem{claim}[thm]{Claim}
	\newtheorem{ass}[thm]{Assumption}
	\newtheorem{prodefi}[thm]{Proposition-Definition}
	\newtheorem{que}[thm]{Question}
	\newtheorem{con}[thm]{Conjecture}
	\newtheorem{exa}[thm]{Example}
	\newtheorem*{assa}{Assumption A}
	\newtheorem*{algstate}{Algebraic form of Theorem \ref{thmstattrainv}}
	
	\newtheorem*{dmlcon}{Dynamical Mordell-Lang Conjecture}
	\newtheorem*{condml}{Dynamical Mordell-Lang Conjecture}
	\newtheorem*{congb}{Geometric Bogomolov Conjecture}

	\newtheorem*{pdd}{P(d)}
	\newtheorem*{bfd}{BF(d)}

	\newtheorem*{probreal}{Realization problems}
	\numberwithin{equation}{section}
	\newcounter{elno}                
	\def\points{\list
		{\hss\llap{\upshape{(\roman{elno})}}}{\usecounter{elno}}}
	\let\endpoints=\endlist
	\newcommand{\SH}{\rm SH}
	\newcommand{\Cov}{\rm Cov}
	\newcommand{\Tan}{\rm Tan}
	\newcommand{\res}{\rm res}
	\newcommand{\Om}{\Omega}
	\newcommand{\om}{\omega}
	\newcommand{\La}{\Lambda}
	\newcommand{\la}{\lambda}
	\newcommand{\mc}{\mathcal}
	\newcommand{\mb}{\mathbb}
	\newcommand{\surj}{\twoheadrightarrow}
	\newcommand{\inj}{\hookrightarrow}
	\newcommand{\zar}{{\rm zar}}
	\newcommand{\Exc}{{\rm Exc}}
	\newcommand{\an}{{\rm an}}
	\newcommand{\red}{{\rm red}}
	\newcommand{\codim}{{\rm codim}}
	\newcommand{\Supp}{{\rm Supp\;}}
		\newcommand{\Leb}{{\rm Leb}}
	\newcommand{\rank}{{\rm rank}}
	\newcommand{\Ker}{{\rm Ker \ }}
	\newcommand{\Pic}{{\rm Pic}}
	\newcommand{\Der}{{\rm Der}}
	\newcommand{\Div}{{\rm Div}}
	\newcommand{\Hom}{{\rm Hom}}
	\newcommand{\im}{{\rm im}}
	\newcommand{\Spec}{{\rm Spec \,}}
	\newcommand{\Nef}{{\rm Nef \,}}
	\newcommand{\Frac}{{\rm Frac \,}}
	\newcommand{\Sing}{{\rm Sing}}
	\newcommand{\sing}{{\rm sing}}
	\newcommand{\reg}{{\rm reg}}
	\newcommand{\Char}{{\rm char\,}}
	\newcommand{\Tr}{{\rm Tr}}
	\newcommand{\ord}{{\rm ord}}
	\newcommand{\diam}{{\rm diam\,}}
	\newcommand{\id}{{\rm id}}
	\newcommand{\NE}{{\rm NE}}
	\newcommand{\Gal}{{\rm Gal}}
	\newcommand{\Min}{{\rm Min \ }}

	\newcommand{\Max}{{\rm Max \ }}
	\newcommand{\Alb}{{\rm Alb}\,}
	\newcommand{\Aff}{{\rm Aff}\,}
	\newcommand{\GL}{{\rm GL}\,}        
	\newcommand{\PGL}{{\rm PGL}\,}
	\newcommand{\Bir}{{\rm Bir}}
	\newcommand{\Aut}{{\rm Aut}}
	\newcommand{\End}{{\rm End}}
	\newcommand{\Per}{{\rm Per}\,}
	\newcommand{\ie}{{\it i.e.\/},\ }
	\newcommand{\niso}{\not\cong}
	\newcommand{\nin}{\not\in}
	\newcommand{\soplus}[1]{\stackrel{#1}{\oplus}}
	\newcommand{\by}[1]{\stackrel{#1}{\rightarrow}}
	\newcommand{\longby}[1]{\stackrel{#1}{\longrightarrow}}
	\newcommand{\vlongby}[1]{\stackrel{#1}{\mbox{\large{$\longrightarrow$}}}}
	\newcommand{\ldownarrow}{\mbox{\Large{\Large{$\downarrow$}}}}
	\newcommand{\lsearrow}{\mbox{\Large{$\searrow$}}}
	\renewcommand{\d}{\stackrel{\mbox{\scriptsize{$\bullet$}}}{}}
	\newcommand{\dlog}{{\rm dlog}\,}    
	\newcommand{\longto}{\longrightarrow}
	\newcommand{\vlongto}{\mbox{{\Large{$\longto$}}}}
	\newcommand{\limdir}[1]{{\displaystyle{\mathop{\rm lim}_{\buildrel\longrightarrow\over{#1}}}}\,}
	\newcommand{\liminv}[1]{{\displaystyle{\mathop{\rm lim}_{\buildrel\longleftarrow\over{#1}}}}\,}
	\newcommand{\norm}[1]{\mbox{$\parallel{#1}\parallel$}}
	\newcommand{\boxtensor}{{\Box\kern-9.03pt\raise1.42pt\hbox{$\times$}}}
	\newcommand{\into}{\hookrightarrow}
	\newcommand{\image}{{\rm image}\,}
	\newcommand{\Lie}{{\rm Lie}\,}      
	\newcommand{\CM}{\rm CM}
	\newcommand{\sext}{\mbox{${\mathcal E}xt\,$}}  
	\newcommand{\shom}{\mbox{${\mathcal H}om\,$}}  
	\newcommand{\coker}{{\rm coker}\,}  
	\newcommand{\sm}{{\rm sm}}
	\newcommand{\pgcd}{\text{pgcd}}
	\newcommand{\trd}{\text{tr.d.}}
	\newcommand{\tensor}{\otimes}
	\newcommand{\hotimes}{\hat{\otimes}}
	\newcommand{\prop}{{\rm prop}}
	\newcommand{\CH}{{\rm CH}}
	\newcommand{\tr}{{\rm tr}}
	\newcommand{\e}{\rm SH}
	
	\renewcommand{\iff}{\mbox{ $\Longleftrightarrow$ }}
	\newcommand{\supp}{{\rm supp}\,}
	\newcommand{\ext}[1]{\stackrel{#1}{\wedge}}
	\newcommand{\onto}{\mbox{$\,\>>>\hspace{-.5cm}\to\hspace{.15cm}$}}
	\newcommand{\propsubset}
	{\mbox{$\textstyle{
				\subseteq_{\kern-5pt\raise-1pt\hbox{\mbox{\tiny{$/$}}}}}$}}
	\newcommand{\sA}{{\mathcal A}}
	\newcommand{\sB}{{\mathcal B}}
	\newcommand{\sC}{{\mathcal C}}
	\newcommand{\sD}{{\mathcal D}}
	\newcommand{\sE}{{\mathcal E}}
	\newcommand{\sF}{{\mathcal F}}
	\newcommand{\sG}{{\mathcal G}}
	\newcommand{\sH}{{\mathcal H}}
	\newcommand{\sI}{{\mathcal I}}
	\newcommand{\sJ}{{\mathcal J}}
	\newcommand{\sK}{{\mathcal K}}
	\newcommand{\sL}{{\mathcal L}}
	\newcommand{\sM}{{\mathcal M}}
	\newcommand{\sN}{{\mathcal N}}
	\newcommand{\sO}{{\mathcal O}}
	\newcommand{\sP}{{\mathcal P}}
	\newcommand{\sQ}{{\mathcal Q}}
	\newcommand{\sR}{{\mathcal R}}
	\newcommand{\sS}{{\mathcal S}}
	\newcommand{\sT}{{\mathcal T}}
	\newcommand{\sU}{{\mathcal U}}
	\newcommand{\sV}{{\mathcal V}}
	\newcommand{\sW}{{\mathcal W}}
	\newcommand{\sX}{{\mathcal X}}
	\newcommand{\sY}{{\mathcal Y}}
	\newcommand{\sZ}{{\mathcal Z}}
	\newcommand{\A}{{\mathbb A}}
	\newcommand{\B}{{\mathbb B}}
	\newcommand{\C}{{\mathbb C}}
	\newcommand{\D}{{\mathbb D}}
	\newcommand{\E}{{\mathbb E}}
	\newcommand{\F}{{\mathbb F}}
	\newcommand{\G}{{\mathbb G}}
	\newcommand{\HH}{{\mathbb H}}
	\newcommand{\LL}{{\mathbb L}}
	\newcommand{\J}{{\mathbb J}}
	\newcommand{\M}{{\mathbb M}}
	\newcommand{\N}{{\mathbb N}}
	\renewcommand{\P}{{\mathbb P}}
	\newcommand{\Q}{{\mathbb Q}}
	\newcommand{\R}{{\mathbb R}}
	\newcommand{\T}{{\mathbb T}}
	\newcommand{\U}{{\mathbb U}}
	\newcommand{\V}{{\mathbb V}}
	\newcommand{\W}{{\mathbb W}}
	\newcommand{\X}{{\mathbb X}}
	\newcommand{\Y}{{\mathbb Y}}
	\newcommand{\Z}{{\mathbb Z}}
	\newcommand{\bk}{{\mathbf{k}}}
	
	\newcommand{\bp}{{\mathbf{p}}}
	\newcommand{\ep}{\varepsilon}
	\newcommand{\bbk}{{\overline{\mathbf{k}}}}
	\newcommand{\Fix}{\mathrm{Fix}}
	
	\newcommand{\conj}{{\rm con}}
	\newcommand{\tor}{{\mathrm{tor}}}
	\renewcommand{\div}{{\mathrm{div}}}
	
	\newcommand{\trdeg}{{\mathrm{trdeg}}}
	\newcommand{\Stab}{{\mathrm{Stab}}}
	
	\newcommand{\OK}{{\overline{K}}}
	\newcommand{\ok}{{\overline{k}}}
	
	\newcommand{\cf}{{\color{red} [c.f. ?]}}
	\newcommand{\jy}{\color{red} jy:}

	\title[]{Local rigidity of Julia sets}
	
\author{Zhuchao Ji}

\address{Institute for Theoretical Sciences, Westlake University, Hangzhou 310030, China}

\email{jizhuchao@westlake.edu.cn}

\author{Junyi Xie}


\address{Beijing International Center for Mathematical Research, Peking University, Beijing 100871, China}

\email{xiejunyi@bicmr.pku.edu.cn}

	
	\date{\today}

	\bibliographystyle{alpha}
	
	\maketitle
	
	\begin{abstract}
	We find criteria ensuring that a local (holomorphic, real analytic, $C^1$) diffeomorphism between the Julia sets of two given rational maps comes from an algebraic correspondence. For example, we show that if there is a local $C^1$-symmetry between the maximal entropy measures of two rational maps, then probably up to a complex conjugation, the two rational maps are dynamically related by an algebraic correspondence. The holomorphic case of our criterion plays an important role in the authors' proof of the Dynamical Andr\'e-Oort conjecture for curves.
		\end{abstract}
	
	
	\tableofcontents
	\section{Introduction}
A classical problem in complex dynamics is to determine when two rational maps of degree at least two have the same Julia set.
 This problem has been studied by many authors \cite{baker1987problem} \cite{beardon1992poiynomials} \cite{dinh2000remarque} \cite{eremenko1989some} \cite{Levin1990} \cite{levin1997two} \cite{schmidt1995polynomials}. In certain situations (e.g., if the Julia set is a circle or $\P^1(\C)$), the shape of the Julia set contains few information.
 So it is also interesting to ask when two rational maps have the same  maximal entropy measure. 	Note that a rational map of degree at least two  has a unique maximal entropy measure. A result of Levin-Przytycki \cite{levin1997two} shows that in this case  the two rational maps are dynamically related by an algebraic correspondence, provided that these two rational maps are not exceptional (see Definition \ref{defiexc}).
Such rigidity issues have recently played important roles in arithmetic dynamics \cite{baker2011preperiodic} \cite{favre2022arithmetic} \cite{ghioca2019dynamical}.

\medskip
 
The first aim of this paper is to study the following more general question:
  \begin{que}\label{querigidityjulia} Can we classify the local (holomorphic, real analytic, $C^1$) diffeomorphisms between the Julia sets (or maximal entropy measures) of two given rational maps of degree at least two?
 \end{que}
 
 Here a diffeomorphisms $h$ between  two measures $\mu_1,\mu_2$ can be understood in a broader sense, asking  either that $h_*(\mu_1)$ is equivalent to $\mu_2$, or $h_*(\mu_1)$ and $\mu_2$ are proportional, see the statement  of Theorem \ref{thmmeasureversion}. This question is inspired by \cite[Problem in page 3]{dujardin2022two} proposed by Dujardin-Favre-Gauthier.
 We expect that under reasonable conditions, such two endomorphisms are \emph{dynamically related} i.e. for endomorphisms $f,g$ on $\P^1(\C)$ of degree at least $2$,
 they are called dynamically related if there exist positive integers $a$ and $b$ and an irreducible algebraic curve $Z$ in $\P^1(\C)\times \P^1(\C)$ such that $Z$ is preperiodic under $(f^a,g^b)$ and dominates both factors. 
 By \cite[Proposition 3.14.]{xie2024algebraicity}, we may take $(a,b)=(\frac{{\rm lcm}\{\deg f,\deg g\}}{\deg f},\frac{{\rm lcm}\{\deg f,\deg g\}}{\deg g}).$ Moreover, we expect that the local diffeomorphism between the Julia sets comes from the algebraic correspondence $Z.$ This local rigidity question was recently studied by Dujardin-Favre-Gauthier \cite{dujardin2022two}and by Luo \cite{luo2021inhomogeneity} in the holomorphic case.

 \begin{rem}\label{remczero}We could not expect a rigidity result for local $C^0$-homeomorphisms.
Indeed for two rational maps in the same stable component,  by $\la$-Lemma \cite[Theorem 4.1]{mcmullen2016complex}, their maximal entropy measures (hence Julia sets) are homeomorphic via a 
quasiconformal homeomorphism.
 \end{rem}

 \medskip

To state our main results, we first introduce some notations. 
Let $g$ be an endomorphisms on $\P^1(\C)$ of degree at least $2$.  
Let $\mu_g$ be the maximal entropy measure for $g$ and
 $\sJ_g$ be the Julia set of $g$. 
 
 \medskip
 
 Denote by $\sigma_{\conj}: \P^1(\C)\to \P^1(\C)$ the anti-holomorphic bijection $z\mapsto \overline{z}.$ 
For every endomorphism $g$ on $\P^1(\C)$ denote by $\overline{g}:\P^1(\C)\to \P^1(\C)$ the endomorphism $z\mapsto \overline{g(\overline{z})}.$
It is clear that $\sJ_{\overline{g}}=\sigma_{\conj}(\sJ_g)$ and we have $\sigma_{\conj}\circ g=\overline{g}\circ\sigma_{\conj}.$
   
\subsubsection{Julia sets of special types}
Recall the following result of Eremenko and Van Strien \cite[Theorem 1 and the first paragraph in page 2]{eremenko2011rational}.
\begin{thm}\label{thmeremen}
There is an open subset $V$ of $\P^1(\C)$ such that 
$\sJ_g\cap V$ is not empty and contained in a $C^1$-smooth curve
if and only if $\sJ_g$ is contained in a circle.
\end{thm}

\begin{defi}
Let $C_g$ be the smallest real analytic subset of $\P^1(\C)$ containing $\sJ_f$. 
\end{defi}
By Theorem \ref{thmeremen}, $C_g$ is $\P^1(\C)$ if $\sJ_g$ is not contained in a circle, otherwise, $C_g$ is the unique circle containing $\sJ_g.$

\medskip

\begin{defi}\label{defismjulia}
We say that $\sJ_g$ is \emph{smooth} if it   equals to $\P^1(\C)$, a circle, or a segment. 
\end{defi}
By Theorem \ref{thmeremen} and \cite[Corollary 4.11]{milnor2011dynamics}, $\sJ_g$ is smooth if and if it is $C^1$-smooth at some point in $\sJ_g.$ 

\medskip

\begin{defi}\label{defiexc}
 As in \cite[Section 1.1]{ji2022homoclinic}, we call $g$ \emph{exceptional} if it is a Latt\`es map or semiconjugates to a monomial map.  Here a rational  map of degree at least 2 is called \emph{Latt\`es} if it is (holomorphically) semi-conjugate to an endomorphism on an elliptic curve. 
\end{defi}
It is clear that when $g$ is exceptional, $\sJ_g$ is smooth and $\mu_g$ is equivalent to the Lebesgue measure on $\sJ_g.$

\subsection{Main results}

In our first result, we answer Question \ref{querigidityjulia} under a measure theoretic assumption.  The holomorphic case of Theorem \ref{measurerigidity}  plays an important role in the authors' proof of the Dynamical Andr\'e-Oort conjecture for curves \cite{ji2023dao}.

For two measures $\mu_1$, $\mu_2$, we write $\mu_1\propto \mu_2$ if they are proportional, i.e. there exists $c>0$ such that $\mu_1=c\mu_2$. Two measures $\mu_1,\mu_2$ are called equivalent if the set of null sets of $\mu_1$ and $\mu_2$ are  the same. A $C^1$ (resp. real analytic) diffeomorphism is a $C^1$ (resp. real analytic) map with $C^1$ (resp. real analytic) inverse.

\begin{thm}[Theorem \ref{measurerigidity}, \ref{measurerigidityrealan} and \ref{measurerigidityrealancone}]\label{thmmeasureversion}
Let $f,g$ be two endomorphisms on $\P^1(\C)$ of degree at least $2$.  Assume that one of them is non-exceptional. 
Let $\mu$ (resp. $\nu$) be a non-atomic invariant ergodic probability measure with positive Lyapunov exponent of $f$ (resp. $g$). 
Let $U\subset \P^1(\C)$ be a connected open subset such that 
$U\cap \, C_g$ is connected and $U\cap \sJ_g\neq \emptyset$. 
Let $h:U\to h(U)\subseteq \P^1(\C)$ be a homeomorphism such that 
\begin{points}
\item $h(U\cap \sJ_g)=h(U)\cap \sJ_f$; if $\sJ_f$ is smooth, we assume further that $h_\ast (\mu_g)\propto \mu_f$ on $h(U)$;
\item  $h_\ast (\nu)$ is equivalent to  $\mu$ on $h(U)$. 
\end{points}
Then the following statements holds:
\begin{points}
\item[{\bf Holomorphic case:}] If $h$ is holomorphic, then there exist positive integers $a$ and $b$ and an irreducible algebraic curve $Z$ in $\P^1(\C)\times \P^1(\C)$ such that $Z$ is preperiodic under $(f^a,g^b)$ and contains the graph $\{(h(x),x), x\in U\}$ of $h$. 

\medskip

\item[{\bf Real analytic case:}] If $h$ is an orientation preserving real analytic diffeomorphism, then there exist positive integers $a$ and $b$ and an irreducible algebraic curve $Z$ in $\P^1(\C)\times \P^1(\C)$ such that $Z$ is preperiodic under $(f^a,g^b)$ and $Z$ contains $\{(h(z),z), z\in U\cap C_f\}$, the graph of $h$.
Moreover, if $\sJ_f$ or $\sJ_g$ is not contained in any circle, then $h$ is holomorphic. 

\medskip

\item[{\bf $C^1$ case:}] If $h$ is a $C^1$-diffeomorphism, then up to change $f$ to its complex conjugation $\overline{f}$, there exist positive integers $a$ and $b$ and an irreducible algebraic curve $Z$ in $\P^1(\C)\times \P^1(\C)$ such that $Z$ is periodic under $(f^a,g^b)$.
Moreover, if $\sJ_f$ or $\sJ_g$ is $\P^1(\C)$, then $h$ is conformal.
\end{points}
\end{thm}

Here a homeomorphism $h:U\to h(U)$ on a open subset $U$ is called {\em conformal} if $h$ is holomorphic or antiholomorphic on every connected component of $U$.
\begin{rem}
One may take $\mu, \nu$ to be $\mu_f$ and $\mu_g$ respectively. Then Theorem \ref{thmmeasureversion} completely answers the measure version of Question \ref{querigidityjulia} in the holomorphic and real analytic cases. Our answer in the $C^1$ case is less precise (c.f. Remark \ref{remcomconj}).
\end{rem}

\begin{rem}
	The holomorphic case of Theorem \ref{measurerigidity} was proved earlier  by Dujardin-Favre-Gauthier \cite{dujardin2022two} under the following additional assumptions: (1) $h$ maps a repelling periodic point of $g$ to a preperiodic point of $f$; (2) $\mu=\mu_f$ and $\nu=\mu_g$; (3) $h_\ast(\nu)$ is strongly equivalent to $\mu$ on $h(U)$, i.e. there exists $C>0$ such that $\frac{1}{C}\mu \leq h_\ast(\nu)\leq C\mu$.
\end{rem}

\medskip

Our second result answers Question \ref{querigidityjulia} under a non-uniform hyperbolic assumption (see condition (ii) in the below theorem) which is satisfied by e.g.  Topological Collet-Eckmann maps \cite[Section 1.1, Theorem E, Lemma 7.2]{przytycki2007statistical} and Weakly Hyperbolic maps in the sense of Rivera Letelier-Shen \cite[Definition 2.1 and 2.2, Corollary 6.3]{rivera2007statistical}. Examples of Weakly Hyperbolic maps in the sense of Rivera Letelier-Shen can be found in  \cite[Fact 2.4]{rivera2007statistical}.

\begin{thm}[Theorem \ref{conicalrigidity}, \ref{conicalrigidityrealan} and \ref{conicalrigiditycone}]\label{conicalrigiditytoget}
Let $f,g$ be two endomorphisms on $\P^1(\C)$ of degree at least $2$.  Assume that one of them is non-exceptional. 
Let $U\subset \P^1(\C)$ be a connected open subset such that 
$U\cap \, C_g$ is connected and $U\cap \sJ_g\neq \emptyset$. 
Let $h:U\to h(U)\subseteq \P^1(\C)$ be a homeomorphism such that
\begin{points}
	\item $h(U\cap \sJ_g)=h(U)\cap \sJ_f$; if $\sJ_f$ is smooth, we assume further that $h_\ast (\mu_g)\propto \mu_f$ on $h(U)$;
	\item the Hausdorff dimension of non-conical points of $g$ is $0$ (c.f. Definition \ref{defconical}).
\end{points}
Then the following statements holds:
\begin{points}
\item[{\bf Holomorphic case:}] If $h$ is holomorphic, then there exist positive integers $a$ and $b$ and an irreducible algebraic curve $Z$ in $\P^1(\C)\times \P^1(\C)$ such that $Z$ is preperiodic under $(f^a,g^b)$ and contains the graph $\{(h(x),x), x\in U\}$ of $h$. 

\medskip

\item[{\bf Real analytic case:}] If $h$ is an orientation preserving real analytic diffeomorphism, then there exist positive integers $a$ and $b$ and an irreducible algebraic curve $Z$ in $\P^1(\C)\times \P^1(\C)$ such that $Z$ is preperiodic under $(f^a,g^b)$ and contains  $\{(h(z),z), z\in U\cap C_f\}$, the graph of $h$. 
Moreover, if $\sJ_f$ or $\sJ_g$ is not contained in any circle, then $h$ is holomorphic. 

\medskip

\item[{\bf $C^1$ case:}] If $h$ is a  $C^1$-diffeomorphism, then up to change $f$ to its complex conjugation $\overline{f}$, there exist positive integers $a$ and $b$ and an irreducible algebraic curve $Z$ in $\P^1(\C)\times \P^1(\C)$ such that $Z$ is periodic under $(f^a,g^b)$.
Moreover, if $\sJ_f$ or $\sJ_g$ is $\P^1(\C)$, then $h$ is conformal.
\end{points}
\end{thm}

\begin{rem}
Since holomorphic maps are real analytic,  the real analytic case of Theorem \ref{thmmeasureversion} and \ref{conicalrigiditytoget} implies the holomorphic case.
But actually the proof of the real analytic case is based on the proof of the holomorphic case. That is the reason why we state the holomorphic case separately. 
\end{rem}

\begin{rem}
In Theorem \ref{thmmeasureversion} and \ref{conicalrigiditytoget},
the assumption of the non-exceptionalness is necessary (see \cite[Remark 2.13]{dujardin2022two}).
\end{rem}

\begin{rem}When $C_g\neq \P^1(\C)$, there are infinitely many local real analytic diffeomorphisms whose  restrictions to $C_g$ are the same as $h$. So
in the real analytic case of Theorem \ref{thmmeasureversion} and \ref{conicalrigiditytoget},
 it is necessary
to  restrict  $h$ to $C_g.$ 
\end{rem}

\begin{rem}\label{remcomconj}If $h$ in the real analytic case of Theorem \ref{thmmeasureversion} (resp. \ref{conicalrigiditytoget}), does not preserve the orientation, then $\sigma_{\conj}\circ h$ preserves the orientation.
So we may apply Theorem \ref{thmmeasureversion} (resp. \ref{conicalrigiditytoget}), for $f,\overline{g}$ and $\sigma_{\conj}\circ h$.
\end{rem}

\begin{rem}The $C^1$ case of Theorem \ref{thmmeasureversion} and \ref{conicalrigiditytoget} is less precise than the holomorphic or the real analytic case. We showed that, probably up to a complex conjugation, $f$ and $g$ are dynamically related. However,
we can not show that the graph of $h$ has intersection with $Z.$ 
\end{rem}

\medskip

Combing the holomorphic case of Theorem \ref{thmmeasureversion} and \cite[Theorem 4.1]{dujardin2022two}, we get the following consequence for polynomial maps.
\begin{cor}\label{corpoly}
Let $f,g$ be two polynomial  endomorphisms on $\P^1(\C)$ of degree at least $2$.  Assume that either $f$ or $g$ is non-exceptional and 
\begin{points}
\item[(1)] either $\sJ_g$ is disconnected;
\item[(2)] or $\sJ_g$ is connected and locally connected.
\end{points}
Let $U\subset \P^1(\C)$ be a connected open subset.
Let $h:U\to h(U)\subseteq \P^1(\C)$ be a biholomorphic map such that $h(U\cap \sJ_g)=h(U)\cap \sJ_f$; if $\sJ_f$ is smooth, we assume further that $h_\ast (\mu_g)\propto \mu_f$ on $h(U)$.
 Then there exist positive integers $a$ and $b$ and an irreducible algebraic curve $Z$ in $\P^1(\C)\times \P^1(\C)$ such that $Z$ is preperiodic under $(f^a,g^b)$ and contains the graph $\{(h(x),x), x\in U\}$ of $h$. 
\end{cor}
\begin{rem}
Corollary \ref{corpoly} was proved by Luo \cite{luo2021inhomogeneity} in the case where $f$ and $g$ are polynomials with the same degree and $\sJ_g$ is connected and locally connected, and it was proved by Dujardin-Favre-Gauthier \cite{dujardin2022two} in the case $g$ is a Topological Collet-Eckmann polynomial (Topological Collet-Eckmann maps automatically satisfy (1) or (2)).
\end{rem}

\medskip

Our another aim is to improve local conjugacies to  algebraic correspondences.
For instance, we prove the rigidity of marked length spectrum of real rational maps. See Section \ref{4} for details.
\begin{thm}\label{interval}
	Let $f,g$ be two endomorphisms on $\P^1(\C)$ of degree at least $2$ defined over $\R$ such that $f$ is non-exceptional.  Let $N_f$ (resp. $N_g$) be $\R\cup\{\infty\}$ or be a compact $f$-invariant (resp. $g$-invariant) interval  contained in $\R\cup\{\infty\}$. Assume 
	\begin{points}
		\item there exists a homeomorphism $h:N_f\to N_g$ such that $h\circ f=g\circ h$ on $N_f$;
		\item $f|_{N_f}$ has positive topological entropy and for every periodic point $x\in N_f$ we have $|df^n(x)|=|dg^n(h(x))|$, where $n$ is the period of $x$.
	\end{points}
	Then  there exists an irreducible algebraic curve $\Gamma\subset \P^1(\C)\times \P^1(\C)$ which is periodic under $(f,g)$, moreover the intersection of $V$ and the graph of $h$ contains a Cantor set.
\end{thm}

Theorem \ref{interval} is closely connected to results in one-dimensional real dynamics by Shub-Sullivan \cite{shub1985expanding}, Martens-de Melo \cite{martens1999multipliers} and Li-Shen \cite{li2006smooth}, where they proved marked length spectrum rigidity for smooth expanding circle maps and for smooth unimodal maps on the interval without periodic or Cantor attractors. 
\par In Section \ref{4}  we also get rigidity results for conformal expanding repeller (CER) (c.f. Theorem \ref{cerrigidity}) and  horseshoe (c.f. Theorem \ref{thmrealanlyhos}).

\subsection{Sketch of the proof}
An important idea in our proofs of Theorem \ref{thmmeasureversion} and \ref{conicalrigiditytoget}
is to study the dynamics at a point $x$ which is ``general" in a certain sense.  This idea has been already used by Levin-Przytycki in their classification of rational maps with identical maximal entropy measures \cite{levin1997two}, and this idea is also crucial in the authors'  proof of the Dynamical Andr\'e-Oort conjecture for curves \cite{ji2023dao}. On the other hand, the strategy in \cite{dujardin2022two} is rather different. They worked on preperiodic points  and the proof involves entire curves and positive currents.

\medskip

Let $f,g$ be two endomorphisms on $\P^1(\C)$ of degree at least $2$.  Assume that one of them is non-exceptional. 
Let $U\subset \P^1(\C)$ be a connected open subset such that 
$U\cap \, C_g$ is connected and $U\cap \sJ_g\neq \emptyset$. 
Let $h:U\to h(U)\subseteq \P^1(\C)$ be a (holomorphic, real analytic, $C^1$) diffeomorphism such that $h(U\cap \sJ_g)=h(U)\cap \sJ_f$; if $\sJ_f$ is smooth, we assume further that $h_\ast (\mu_g)\propto \mu_f$ on $h(U)$.

\medskip

If one of $f,g$ is exceptional, then $\sJ_f$ and $\sJ_g$ are smooth. 
Then by \cite[Theorem 1]{zdunik1990parabolic}, both of them are exceptional. 
So both $f$ and $g$ are non-exceptional.

\medskip

We define a point in $U\cap \sJ_g$ to be \emph{bi-conical} for $(g,h,f)$ if 
 it is not $g-$preperiodic,  and there are positive constants $r,R,K$ and two sequences of positive integers $n_j\to +\infty$, $m_j\to +\infty$, $j\geq 1$ such that 
\begin{points}
\item $g^{n_j}:W_j\to B(g^{n_j}(x),r)$ is injective and $W_j\subset U$, where $W_j$ is the connected component of $g^{-n_j}(B(g^{n_j}(x),r))$ containing $x$;
	\item the map $h_j:=f^{m_j}\circ h\circ g_{n_j}:B(g^{n_j}(x),r)\to \P^1(\C)$  is injective  and satisfies 
	\begin{equation*}
		B(f^{m_j}(h(x)),R/K)\subset h_j(B(g^{n_j}(x),r))\subset B(f^{m_j}(h(x)),R),
	\end{equation*}
	where $g_{n_j}$ is the inverse map of $g^{n_j}:W_j\to B(g^{n_j}(x),r)$.
\end{points}

\medskip

Our first step is to show the existence of bi-conical points for $(g,h,f)$. 
Under condition (ii) of Theorem \ref{conicalrigiditytoget},  this is easy to show
via the constructions of CERs. Under condition (ii) of Theorem \ref{thmmeasureversion}, the proof is more delicate. 
In Lemma \ref{key}, we show that $\nu$-a.e. point in $U$ is bi-conical for $(g,h,f)$ via ergodic theory.
Moreover, the set of times $G=\{n_i, i\geq 0\}$ can be asked to have positive lower density. This is done in Section \ref{3}.

\medskip

We first prove the holomorphic case in Section \ref{3}, which is the foundation of the real analytic and the $C^1$ case.
Assume the existence of a bi-conical point $x$ for $(g,h,f)$. By Montel's theorem, we can construct a normal family $\{h_i: D\to \P^1(\C)\}$ of injective holomorphic maps on some disk $D$ with $D\cap \sJ_g\neq \emptyset$ and satisfy condition (i) of Theorem \ref{thmmeasureversion} (and Theorem \ref{conicalrigiditytoget}).  Moreover, every limit of this family is non-constant. Applying Levin's result \cite[Theorem 1]{Levin1990} (see also Theorem \ref{thmlevin}), we construct two horseshoes in $\sJ_g$ and $\sJ_f$ respectively and a holomorphic
conjugacy between them. Applying a generalized form (c.f. Theorem \ref{inou1}) of a theorem of Inou \cite[Theorem 2]{inou2011extending} based on Eremenko's theorem \cite[Theorem 2]{eremenko1989some} (c.f. Theorem \ref{eremenko}), we show that the conjugacy can be improved to an algebraic correspondence. This implies the holomorphic case of Theorem \ref{thmmeasureversion} and \ref{conicalrigiditytoget}. Generalized Inou's theorem is discussed in Section \ref{2}.

\medskip

There are three ingredients in the above argument which not apply to the 
real analytic case directly. The first one is Montel's theorem.  We replace it by Lemma \ref{lemconenormal}.
The second one is Levin's result \cite[Theorem 1]{Levin1990}. For this, we replace it by Theorem \ref{thmlevinrealan}, which generalizes \cite[Theorem 1]{Levin1990} to the real analytic setting. The third one is Theorem \ref{inou1}. We replace it by our rigidity result for horseshoes (c.f. Theorem \ref{thmrealanlyhos}). This is done in  Section \ref{5} and \ref{6}.

\medskip

To prove the $C^1$ case, our idea is to reduce it to the real analytic case. From a bi-conical point $x$ for $(g,h,f)$, we construct a sequence $\{h_i: D\to \P^1(\C)\}$ of injective $C^1$-maps on some disk $D$ with $D\cap \sJ_g\neq \emptyset$ converging to a real analytic map $H$ which 
satisfyies condition (i) of Theorem \ref{thmmeasureversion} (and Theorem \ref{conicalrigiditytoget}).
This reduces the $C^1$ case of Theorem \ref{conicalrigiditytoget} to the real analytic case. 
The $C^1$ case of Theorem \ref{thmmeasureversion} is more delicate, since condition (ii) of  Theorem \ref{thmmeasureversion} may not be satisfied by $H.$
To solve this problem, we need the fact that the times set $G=\{n_i, i\geq 0\}$ has positive lower density.  We construct the sequence $\{h_i: D\to \P^1(\C)\}$ based on a chosen subsequence $n_{i_j}, j\geq 0$ of $G.$ One show that, under this choice, $g^{n_{i,j}}(x)$ converges to a point $p\in D$ which is bi-conical for $(g,H,f)$ (c.f. Lemma \ref{lempbicon}).
Then the argument in the real analytic case can be applied to $(p,g,H,f)$, which  concludes the proof.
This is done in Section \ref{7}. 

\subsection*{Acknowledgement}
We would like to thank Donyi Wei for his proof of Lemma \ref{lemconenormal}.
\par The first-named author would like to thank Beijing International Center for Mathematical Research in Peking University for the invitation.  The
first named author Zhuchao Ji is supported by ZPNSF Grant (No. XHD24A0201) and by NSFC Grant (No. 12401106). The second-named author Junyi Xie is supported by NSFC Grant (No.12271007).

	\section{A generalized Inou's theorem}\label{2}
	The aim of this section is to prove the following theorem, which will be used in the proof of Theorem \ref{measurerigidity}.  
	\begin{thm}\label{inou}
	Let $f,g$ be two non-Latt\`es endomorphisms on $\P^1(\C)$ of degree at least $2$. Assume that
	\begin{points}
		\item there exists a connected open set $U\subset \P^1(\C)$ and an open subset $U'\subset U$ such that $f(U')\subset  U$, $f:U'\to U$ is non-injective, and $f:U'\to U$ has a repelling fixed point;
		\item there exists a non-constant holomorphic map $h:U\to \P^1(\C)$ such that $h\circ f=g\circ h$ on $U'$.
	\end{points}
Then there exists an irreducible algebraic curve $\Gamma\subset \P^1(\C)\times \P^1(\C)$ which is invariant under $(f,g)$ and contains the graph of $h$.
	\end{thm}
	
In \cite{inou2011extending}, Inou proved the above result when $f:U'\to U$ is polynomial-like. His proof indeed works in the general case. 
We give a proof of Theorem \ref{inou} for the convenience of the readers.
The main ingredients of the proof are Inou's construction \cite[Theorem 2]{inou2011extending} and Eremenko's theorem \cite[Theorem 2]{eremenko1989some} as follows:

\begin{thm}[Inou]\label{inou1}
	Let $f,g$ be two endomorphisms on $\P^1(\C)$. Let $U\subset \P^1(\C)$ be an open set. Let $h:U\to \P^1(\C)$ be a non-constant holomorphic map. Let $\Gamma_0:=\left\{(z,h(z)):z\in U\right\}\subset \P^1(\C)\times \P^1(\C)$ be the graph of $h$. Consider the forward invariant  set 
	\begin{equation*}
		\Gamma:=\bigcup_{n\geq 0} F^n(\Gamma_0)
	\end{equation*}
 of the product map $F:=(f,g)$.  Then there exists a Riemann surface $X$ and holomorphic maps $G:X\to X$ and $\pi:X\to \P^1(\C)\times \P^1(\C)$ such that $\pi(X)=\Gamma$, $\pi\circ G=F\circ \pi$ on $X$ and has the following properties:
	\begin{points}
	\item Let $\phi_i:=p_i\circ \pi$, where $p_i:\P^1(\C)\times \P^1(\C)\to \P^1(\C)$ is the projection to the $i$-th  coordinate. Then there exists an open subset $V\subset X$ such that $\pi(V)=\Gamma_0$ and $\phi_1:V\to U$ is biholomorphic.
	\item The  cardinality of the preimages of a point in $X$ by $G$ is not greater than $\deg f\cdot \deg g$.
	\item If $U$ is connected and there exists $U'\subset U$ such that $f(U')\subset U$, and $h\circ f=g\circ h$ on $U'$, then $X$ is connected, $\phi_1\circ G=f\circ \phi_1$ on $X$ and  $\phi_2\circ G=g\circ \phi_2$ on $X$.
\end{points}
\end{thm}
\medskip
\begin{thm}[Eremenko]\label{eremenko}
	Let $R,G$ be two endomorphisms on $\P^1(\C)$ such that $\deg G\geq 2$. Let $\phi$ be a meromorphic function defined on $\C$ or  $\C^\ast$ with an essential singularity at $\infty$. If $\phi$ satisfies $\phi\circ G=R\circ \phi$, then $R$ is a Latt\`es map and $G$ is exceptional of monomial type.

\end{thm}
\medskip

\par \proof[Proof of Theorem \ref{inou}]
We first  construct $\Gamma_0, \Gamma, X, G, \pi_i, \phi$ and $V$ as in Theorem \ref{inou1}.   We only need to show that $\Gamma=\pi(X)$ is an irreducible algebraic curve in $\P^1(\C)\times \P^1(\C)$. Once this holds, since $\Gamma$ contains the graph of $h$, it is invariant under $(f,g)$.
\par  By the assumptions (i) and (ii) of Theorem \ref{inou} and by Theorem \ref{inou1}, $X$ is connected, and $G:X\to X$ is non-injective and has a repelling fixed point. Hence $X$ can not be a hyperbolic Riemann surface. Therefore $X$ is biholomorphic to $\P^1(\C)$, a torus $\C/\La$, $\C$ or $\C^\ast$. 
\par In the first two cases $X$ is projective, hence $\phi_1$ and $\phi_2$ are finite morphisms, 
hence $\Gamma=\pi(X)$ is an irreducible algebraic curve in $\P^1(\C)\times \P^1(\C)$. In the case $X=\C$ or $\C^\ast$,  by Theorem \ref{inou1} (ii), $G$ is a polynomial map in the case $X=\C$, and $G(z)=z^m$ for some $m\in\Z$ in a suitable coordinate in the case $X=\C^\ast$.  
By assumption (ii) in Theorem \ref{inou}, since $f$ and $G$ are conjugated by $\phi_1$ on $U'$, $G$ is non-injective, i.e. $G$ has degree at least $2$.
By Theorem \ref{inou1} (iii), we have $\phi_1\circ G=f\circ \phi_1$ on $X$ and  $\phi_2\circ G=g\circ \phi_2$ on $X$.  Since $f$ and $g$ are not Latt\`es, by Theorem \ref{eremenko}, $\phi_1$ and $\phi_2$ must be finite morphisms. Hence $\Gamma=\pi(X)$ is an irreducible algebraic curve in $\P^1(\C)\times \P^1(\C)$. This finishes the proof.
\endproof

\section{Variants of Levin's theorem}\label{5}
Let $f$ be an endomorphism of $\P^1(\C)$ of degree at least $2$. 
We reformulate Levin's theorem \cite[Theorem 1]{Levin1990} in the following form. 
\begin{thm}\label{thmlevin}
Let $U$ be a connected open subset of $\P^1(\C)$ with $U\cap \sJ_f\neq \emptyset.$
Let $\sigma_n: U\to \P^1(\C), n\geq 0$ be a a family of injective holomorphic maps satisfying $\sigma_n\to \id$ as $n\to \infty$ and 
$\sigma_n^{-1}(\sJ_f)=\sJ_f\cap U$. If in additional $\sJ_f$ is smooth, we assume further that  $(\sigma_n)^\ast (\nu_f)\propto \mu_f|_U$.
Then either $f$ is exceptional or $\sigma_n=\id$ for $n$ sufficiently large.
\end{thm}

This theorem will be used in our proofs of the holomorphic case of Theorem \ref{thmmeasureversion} and Theorem \ref{conicalrigiditytoget}.
The aim of this section is to prove the following result which generalizes Theorem \ref{thmlevin} to real analytic setting.

\begin{thm}\label{thmlevinrealan}
Let $U$ be a open subset of $\P^1(\C)$ with $U\cap \sJ_f\neq \emptyset.$ Assume that $U\cap C_f$ is connected.
Let $\sigma_n: U\to \P^1(\C), n\geq 0$ be a family of injective real analytic maps satisfying $\sigma_n\to \id$ as $n\to \infty$ in $C^1$-topology and 
$\sigma_n^{-1}(\sJ_f)=\sJ_f\cap U$. If in additional $\sJ_f$ is $\P^1(\C)$, we assume further that  $(\sigma_n)^\ast (\mu_f)\propto \mu_f|_U$.
Then $f$ is exceptional or $\sigma_n|_{C_f}=\id$ for $n$ sufficiently large.
\end{thm}

Theorem \ref{thmlevinrealan} will be used in our proof of the real analytic and $C^1$ cases of Theorem \ref{thmmeasureversion} and Theorem \ref{conicalrigiditytoget}.
Our proof of Theorem \ref{thmlevinrealan}  follows Levin's original  strategy, but  more elementary in each step.
\subsection{Movable case}
As in Levin's original proof of Theorem \ref{thmlevin}, we first treat the case where there is a repelling periodic point $o$ which is not fixed by any $\sigma_n$. 
 In this step, we only need $\sigma_n, n\geq 0$ to be $C^1$. Our proof follows Levin's original strategy in the holomorphic case.

\begin{thm}\label{thmlevinnf}
Let $U$ be a connected open subset of $\P^1(\C)$ with $U\cap \sJ_f\neq \emptyset.$
Let $\sigma_n: U\to \P^1(\C), n\geq 0$ be a family of injective $C^1$-maps satisfying $\sigma_n\to \id$ as $n\to \infty$ in $C^1$-topology and 
$\sigma_n^{-1}(\sJ_f)=\sJ_f\cap U$. If in additional $\sJ_f$ is smooth, we assume further that  $(\sigma_n)^\ast (\mu_f)\propto \mu_f|_U$.
If there is a repelling fixed point $o$ of $f$ such that $\sigma_n(o)\neq o$ for every $n\geq 0$, then $f$ is exceptional.
\end{thm}

\proof
After shrinking $U$, we may ask $U$ to be a linearization domain of $o$. 
In a suitable coordinate $z$, we may let $U=\{|z|\leq 1\}$,   $o$ be the origin and $f$ is $z\mapsto \la z$ with $|\la|>1$. Write $\sigma_n=a_n+b_n(z)+\epsilon_n(z)$
where $b_n\in M_{2\times 2}(\R), \epsilon_n(0)=0$ and $d\epsilon_n(0)=0$. 
Our assumption shows that $a_n\neq 0$ for every $n\geq 0$, $a_n\to 0$, $b_n\to \id$ and $\epsilon_n\to 0$ in $C^1$-topology. 
After shrinking $U$, we may assume that there are sequences $c_n,d_n>0$ tending to $0$ such that 
$|\epsilon_n(z)|\leq c_n$ and $|d\epsilon_n(z)|\leq d_n$ on $U.$

\medskip

Let $U_1:=\{|z|<0.9\}\subseteq U.$  For every complex number $c$ with $|c|<0.1$, define $T_c: U_1\to U$ by $z\mapsto z+c.$
We say that $c$ is \emph{good}  if the following holds:
\begin{points}
\item[$\d$] $T_{c}^{-1}(\sJ_f)=\sJ_f\cap U_1$;
\item[$\d$] if in additional $\sJ_f$ is smooth, we have that  $$T_{c}^\ast (\mu_f)\propto \mu_f|_{U_1}.$$
\end{points}
It is clear that the set $\sG$ of good $c$ in $\{|c|<0.1\}$ forms a closed subset in $\{|c|<0.1\}$.
Moreover, for $c_1,c_2\in \C$ with $\max\{|c_1|, |c_2|, |c_1+c_2|\}<0.1$, if both $c_1$ and $c_2$ are good, $c_1+c_2$ is good.

\medskip

After taking subsequence, we may assume that there is a sequence $l_n\geq 0$ tending to $+\infty$ such that $a_n\la^{l_n}\to a$ and $|a|\in (0, 0.1).$ 
View $\la$ as a matrix in $M_{2\times 2}(\R)$.
Fix any $m\geq 0$, define 
\begin{equation*}
\begin{split}
\delta_{n,m}(z):&=\la^{l_n-m}\sigma_n(\la^{-l_n+m}z)\\
&=a_n\la^{l_n-m}+\la^{l_n-m}b_n \la^{-l_n+m}(z)+\la^{l_n-m}\epsilon(\la^{-l_n+m}z).
\end{split}
\end{equation*}
We have $a_n\la^{l_n-m}\to a/\la^m.$ 
Morever $$\tr(\la^{l_n-m}b_n \la^{-l_n+m})=\tr (b_n)\to 2, \det(\la^{l_n-m}b_n\la^{-l_n+m})=b_n\to 1$$ and the conformal index of $\la^{l_n}b_n\la^{-l_n}$ is the same as 
$b_n$, which tends to $1.$  Here for a matrix $A\in M_2(\R)$, the conformal index of $A$ is the radius $\mu_2/\mu_1$ where $\mu_1\geq \mu_2$ are the singular values of $A$. So we get $\la^{l_n-m}b_n \la^{-l_n+m}$ tends to $\id.$ 
We have $$|\la^{l_n-m}\epsilon(\la^{-l_n+m}z)|\leq |\la^{l_n-m}\la^{-l_n+m}d_n|$$ which tends to $0.$
So $\delta_{n,m}$ tends to the map $T_{a/\la^m}: z\mapsto z+a/\la^m$ locally uniformly as $n\to \infty$. 
Then $a/\la^m$ is good for every $m\geq 0.$

We claim that for every $m\geq 0$ and $r\in [0,1)$, $ra/\la^m$ is good. Pick a sequence $m_i\to \infty$ such that $$\lim_{i\to \infty}(\la/|\la|)^{m_i}\to 1.$$
As $a/\la^{m+m_i}$ is good, $\lfloor r|\la|^{m_i}\rfloor a/\la^{m+m_i}$ is good. As $$\lim_{i\to \infty}\lfloor r|\la|^{m_i}\rfloor a/\la^{m+m_i}=ra/\la^m,$$ $ra/\la^m$ is good.

\medskip

We first assume that  $\la\not\in \R$. Then for every $r_1, r_2\in [0,1/2)$, $r_1a+r_2a/\la$ is good.
Then $\sG$ contains a non-empty open subset of $\C.$ This implies that $\sJ_f=\P^1(\C)$ and  $\mu_f|_{U_1}$ is the Lebesgue measure. By Zdunik's \cite[Theorem 1]{zdunik1990parabolic}, $f$ is exceptional. 

Now assume that $\la\in \R$. Note that $T_{a/\la^n}$ tends to $\id$.
If there is a repelling periodic point $o_1$ in $U_1$ having a non-real multiplier. 
Note that $T_{a/\la^n}(o_1)\neq o_1$ for every $n\geq 1.$
After replacing $U$ by $U_1$, $o$ by $o_1$ and $\sigma_n$ by $T_{a/\la^n}$, the previous paragraph implies that $f$ is exceptional. 
So we may assume that every repelling periodic point $U_1$ has real multiplier. By \cite[Theorem 1 and the first paragraph below Corollary 1]{eremenko2011rational}, $\sJ_f$ is contained in a circle $C_f$. As $C_f$ is $f$ invariant, in $U$ it is a line passing through $0$ in the coordinate $z$.
We may assume that $C_f\cap U=\{z\in \R\}$. Then $\sG\subseteq \R.$ As $ra$ is good is good for every $r\in [0,1)$, $\sG$ contains a non-empty open subset of $\R$. So $\mu_f|_{U_1}$ is the Lebesgue measure on $\sJ_f\cap {U_1}$. By Zdunik's \cite[Theorem 1]{zdunik1990parabolic}, $f$ is exceptional. 
\endproof

\subsection{General case}
\proof[Proof of Theorem \ref{thmlevinrealan}]
Assume $\sigma_n|_{C_f}\neq\id$ for infinitely many $n$, we need to show  that $f$ is exceptional. After taking subsequence we assume that $\sigma_n|_{C_f}\neq \id$ for every $n\geq 0$. After shrinking $U$ and pasing to an iteration of $f$, we may ask $U$ to be a linearization domain of $o$, where $o$ is a repelling fixed point of $f$.
In a suitable coordinate $z$, we may let $U=\{|z|\leq 1\}$, $o$  be the origin and $f$ is $z\mapsto \la z$ with $|\la|>1$. 
As $C_f$ is $f$ invariant, in $U$ it is a line passing through $0$ in the coordinate $z$.
We may assume that $C_f\cap U=\{z\in \R\}$.
By Theorem \ref{thmlevinnf}, we may assume that $\sigma_n(0)=0$ for all $n\geq 0.$

\medskip

We first treat the case where there is $i\geq 0$ such that $d\sigma_i(0)=\id.$ Let $\delta_n$ be the map $z\mapsto \la^n\sigma_i(\la^{-n}z)$.
Then $\delta_n\to \id$ in  $C^1$-topology. Since $\sigma_i\neq \id$, we have $\delta_n\neq \id.$
For every repelling periodic point $p\in U$, let $N_p\subseteq \Z_{\geq 0}$ be the set of $n\geq 0$ such that $\delta_n(p)=p.$
If there is a repelling periodic point $p\in U$ such that $\Z_{\geq 0}\setminus N_p$ is infinite, then we conclude the proof by Theorem \ref{thmlevinnf} for 
$\delta_n, n\in \Z_{\geq 0}\setminus N_p$ and $p.$ Now we may assume that $\Z_{\geq 0}\setminus N_p$ is finite for every repelling periodic point $p\in U$.
For every repelling periodic point $p\in U$ and $n\in N_p$, we have $\la^{-n}p\in \Fix(\delta_0).$ 
We claim that $\Fix(\delta_0)$ contains the line $p\R\cap U$ for every repelling periodic point $p\in U.$ 
In particular,  we get $\delta_0(p)=p$.
 Since $\delta_0$ is real analytic and $\delta_0\neq \id$, the set $\Fix(\delta_0)$ is a  proper and real analytic closed subset of $U$, this implies $\la^r\in \R$ for some $r\geq 1$. 
Then $\Fix(\delta_0)$ contains the line $p\R\cap U$.  
This implies the claim.

We apply the claim for every repelling periodic points in $U$.  
Since repelling periodic points are dense in $\sJ_f$, $\sJ_f\cap U\subseteq \Fix(\delta_0).$ Since $\Fix(\delta_0)$ is a proper and real analytic closed subset of $U$,by Theorem \ref{thmeremen}, $C_f$ is a circle. 
Let $p$ be a repelling periodic point in $U\setminus 0.$
Then $p\in \R$ in the coordinate $z$. The claim implies that $(C_f\cap U)\subseteq \Fix(\delta_0).$
In other words $\delta_0|_{C_f}=\id$, which is a contradiction.

\medskip

We now assume that $d\sigma_i(0)\neq \id$ for every $i\geq 0.$ Note that for every $i\geq 0$, there is a sequence $n_j\to \infty$
such that the maps $z\mapsto \la^{n_j}\sigma_i(\la^{-n_j}z)$ tend to $d\sigma_i(0)$.
After replacing $\sigma_i$ by $d\sigma_i(0)$, we may assume that all $\sigma_i$ are $\R$-linear.
If $C_f$ is a circle, then $C_f\cap U=U\cap \R.$
Since $\sigma_i|_{C_f}\neq \id$ and $\sigma_i$ is $\R$-linear, $\Fix(\sigma_i)\cap C_f=\{0\}.$
Pick a repelling periodic point $p\in U\setminus \{0\}.$ We have $\sigma_i(p)\neq p.$ We conclude the proof by Theorem \ref{thmlevinnf}.
We may assume that $C_f=\P^1(\C)$. By Theorem \ref{thmeremen}, there are repelling periodic points $p,q\in U$ such that 
$0, p, q$ are not collinear. Then $\{p,q\}\not\subseteq \Fix(\sigma_i)$ for every $i\geq 0.$
We may assume that for infinitely many $i\geq 0$, $p\not\in \Fix(\sigma_i).$ We then conclude the proof by Theorem \ref{thmlevinnf}.
\endproof

\section{Holomorphic  local rigidity of Julia sets}\label{3}
The aim of this section is to prove the following two theorems, which are the holomorphic case of Theorem \ref{thmmeasureversion} and Theorem \ref{conicalrigiditytoget} respectively.
\begin{thm}\label{measurerigidity}
Let $f,g$ be two endomorphisms on $\P^1(\C)$ of degree at least $2$.  Assume that one of them is non-exceptional. Let $\mu$ (resp. $\nu$) be a non-atomic invariant ergodic probability measure with positive Lyapunov exponent of $f$ (resp. $g$). Let $U\subset \P^1(\C)$ be a connected open subset such that $U\cap \Supp \nu\neq \emptyset$.  Let $h:U\to h(U)\subseteq \P^1(\C)$ be a biholomorphic map such that 
\begin{points}
\item $h(U\cap \sJ_g)=h(U)\cap \sJ_f$; if $\sJ_f$ is smooth, we assume further that $h_\ast (\mu_g)\propto \mu_f$ on $h(U)$;
\item  $h_\ast (\nu)$ is equivalent to  $\mu$ on $h(U)$. 
\end{points}
Then there exist positive integers $a$ and $b$ and an irreducible algebraic curve $Z$ in $\P^1(\C)\times \P^1(\C)$ such that $Z$ is preperiodic under $(f^a,g^b)$ and contains the graph $\{(h(x),x), x\in U\}$ of $h$. 
\end{thm}

\medskip

\begin{thm}\label{conicalrigidity}
Let $f,g$ be two endomorphisms on $\P^1(\C)$ of degree at least $2$. 
Assume that one of them is non-exceptional .
 Let $U\subset \P^1(\C)$ be a connected open subset and let $h:U\to h(U)\subseteq \P^1(\C)$ be a biholomorphic map. Assume that
\begin{points}
	\item $h(U\cap \sJ_g)=h(U)\cap \sJ_f$; if $\sJ_f$ is smooth, we assume further that $h_\ast (\mu_g)\propto \mu_f$ on $h(U)$;
	\item the Hausdorff dimension of non-conical points of $g$ is $0$.
\end{points}
\par Then there exists positive integers $a$ and $b$ and an irreducible algebraic curve $Z$ in $\P^1(\C)\times \P^1(\C)$ such that $Z$ is preperiodic under $(f^a,g^b)$ and contains the graph $\{(h(x),x), x\in U\}$ of $h$. 
\end{thm}

\subsection{Bi-conical points}
\begin{defi}\label{defconical}
	Let $g$ be an endomorphisms on $\P^1(\C)$ of degree at least $2$.  A point $x\in \sJ_g$ is called conical if there exists $r>0$ and a sequence of positive integers $n_j\to +\infty$ such that 
	$$g^{n_j}:W_j\to B(g^{n_j}(x),r)$$
	is injective, where $W_j$ is the connected component of $g^{-n_j}(B(g^{n_j}(x),r))$ containing $x$.
\end{defi}
\medskip
\begin{lem}\label{derivative}
	In the setting of Definition \ref{defconical}  we have
	$$|dg^{n_j}(x)|\to +\infty.$$
\end{lem}
\begin{proof}
	Assume by contradiction that by passing to a subsequence of $\left\{n_j\right\}$  there exists $M>0$ such that
	$|dg^{n_j}(x)|<M$.   Shrink $r$ if necessary such that $2r<\diam (\sJ_g)$. By Koebe one-quarter theorem $W_j$ contains the  disk $D:=B(x,r/(4M))$. This implies $g^{n_j}(D)\subset B(g^{n_j}(x),r)$ for every $j\geq 1$, contradicts to the fact $\sJ_g\subset g^{n_j}(D)$ when $n_j$ large enough. 
\end{proof}

Let $f,g$ be two endomorphisms on $\P^1(\C)$ of degree at least $2$.  Let $U\subset \P^1(\C)$ be an open subset and let $h:U\to h(U)\subseteq \P^1(\C)$ be a homeomorphism.
\begin{defi}\label{defibiconical}A point $x\in U\cap \sJ_g$ is called \emph{bi-conical} for $(g,h,f)$ if 
 it is not $g$-preperiodic,  for which there are positive constants $r,R,K$ and two sequences of positive integers $n_j\to +\infty$, $m_j\to +\infty$, $j\geq 1$ such that 
\begin{points}

	\item $g^{n_j}:W_j\to B(g^{n_j}(x),r)$ is injective and $W_j\subset U$, where $W_j$ is the connected component of $g^{-n_j}(B(g^{n_j}(x),r))$ containing $x$;
	\item the map $h_j:=f^{m_j}\circ h\circ g_{n_j}:B(g^{n_j}(x),r)\to \P^1(\C)$  is injective  and satisfies 
	\begin{equation*}
		B(f^{m_j}(h(x)),R/K)\subset h_j(B(g^{n_j}(x),r))\subset B(f^{m_j}(h(x)),R),
	\end{equation*}
	where $g_{n_j}$ is the inverse map of $g^{n_j}:W_j\to B(g^{n_j}(x),r)$;
\end{points}
\end{defi}

If $x$ is bi-conical for $(g,h,f)$,  then for every small neighborhood $V$ of $h(x)$, there exists $m
\geq 1$ such that $B(f^m(h(x)),R/K)\subset f^m(V)$, this implies that $h(x)\in \sJ_f$.
\begin{rem}
For $n,m\geq 0$, if 
$g^{n}:W\to B(g^{n}(x),r)$ is injective and $W\subset U$, where $W$ is the connected component of $g^{-n}(B(g^{n}(x),r))$ containing $x$ and 
$f^m$ is injective on $h(W)$, then $g^n(x)$ is bi-conical for $f^m\circ h\circ g_n$, where  $g_{n}$ is the inverse map of $g^{n}:W\to B(g^{n}(x),r)$.
\end{rem}

\subsection{Holomorphic rigidity via bi-conical points}
\begin{lemma}\label{conical}
Let $f,g$ be two non-exceptional endomorphisms on $\P^1(\C)$ of degree at least $2$.  Let $U\subset \P^1(\C)$ be a connected open subset and let $h:U\to h(U)\subseteq \P^1(\C)$ be a biholomorphic map such that 
 $h(U\cap \sJ_g)=h(U)\cap \sJ_f$,  if $\sJ_f$ is smooth, we assume further that $h_\ast (\mu_g)\propto \mu_f$ on $h(U)$. 
 Assume that there is a point $x\in U\cap \sJ_g$ which is bi-conical for $(g,h,f)$, then
 Then there exists positive integers $a$ and $b$ and an irreducible algebraic curve $Z$ in $\P^1(\C)\times \P^1(\C)$ such that $Z$ is preperiodic under $(f^a,g^b)$ and contains the graph $\{(h(z),z), z\in U\}$ of $h$. 
\end{lemma}
\begin{proof}
	By Lemma \ref{derivative} we have 
	\begin{equation}\label{iii}
|dg^{n_j}(x)|\to +\infty.
	\end{equation}
\par Shrink $r$ if necessary, we may further assume that there is $C>1$ such that for every $j\geq 1$ we have
	\begin{equation*}
	C^{-1}<|dh_j(x)/dh_j(y)|<C
\end{equation*}
for every $y\in B(g^{n_j}(x),r)$.
By passing to a subsequence of $\{(n_j, m_j)\}$, 
we may assume that $g^{n_j}(x)$ converges.  Then there exists a disk $D'$ centered at the limit point and with radius $r/4$. We may assume that $g^{n_j}(x)\in D'$ for every $j\geq 1$.  Let  $D$ be the disk of radius $r/2$ and has  the same center with $D'$. Then we have $$B(g^{n_j}(x),r/4)\subset D\subset B(g^{n_j}(x),r)$$ for every $j\geq 1$. 
By Montel's theorem, $\{h_j|_D\}$ is a normal family and every limit map in this family is non-constant. 
Cover $\sJ_f$ by finitely many disks of radius $R/(8CK)$.  By passing to a subsequence of $\{(n_j, m_j)\}$, there exists a disk $V$ in this finite family such that $f^{m_j}(h(x))\in V$ for every $j\geq 1$. By (ii) for every $j\geq 1$, we have $$V\subset B(f^{m_j}(h(x)),R/(4CK)) \subset h_j(B(g^{n_j}(x),r/4))\subset h_j(D).$$
 For every $j\geq 1$ the following maps are well defined  $$\sigma_j:=h_{j}\circ h_{1}^{-1}|_{V}:V\to \P^1(\C).$$
By our construction $\sigma_j$  are injective holomorphic maps with bounded distortion, moreover the diameter of $\sigma_j(V)$ are uniformly bounded from above and from below.   Then $\left\{\sigma_j\right\}$ is a normal family and every limit map in this family is non-constant.  By our construction,  we have $\sigma_j(V\cap \sJ_f)=\sigma_j(V)\cap \sJ_f$, i.e. $\sigma_j$ are local symmetries of $\sJ_f$. When $\sJ_f$ is smooth, using the total invariance of maximal entropy measures, for every $j\geq 1$ we have $(h_j)_\ast (\nu)\propto \mu$ on $h_j(D)$, where  $\nu$ and $\mu$ are maximal entropy measures of $g$ and $f$. As a consequence we have $(\sigma_j)_\ast (\mu)\propto \mu$ on $\sigma_j(V)$.  Now by Levin's result \cite[Theorem 1]{Levin1990} (see also Theorem \ref{thmlevin}), since $f$ is non-exceptional,  $\left\{\sigma_j\right\}$ is a finite set.  By passing to a subsequence we may assume that $\sigma_j=\sigma_{j_0}$ on $V$ for every $j\geq j_0$. This implies $h_j=h_{j_0}$ on $D$ for every $j\geq j_0$.  Without loss of generality we may assume $j_0=1$. 
\medskip
\par  The definition of $h_j, j\geq 1$ shows that $$f^{m_j}\circ h\circ g_{n_j}=f^{m_{1}}\circ h\circ g_{n_{1}}$$ on $D$.  
By (\ref{iii}), we have $m_j\to +\infty$. By passing to a subsequence, we may assume $n_j>\max (n_1, n_2)$ and $m_j>\max (m_1, m_2)$ for every $j>2$. For $j>2$ let $U_j$ be the connected component of $g^{-(n_{j}-n_1)}(D)$ containing $g^{n_1}(x)$. Thus we have
\begin{equation*}
	f^{m_j-m_1}\circ h_1=h_1\circ g^{n_j-n_1}
\end{equation*}
on $U_j$. 
Similarly  for $j>2$ let $U'_j$ be the connected component of $g^{-(n_j-n_2)}(D)$ containing $g^{n_2}(x)$. Thus we have
\begin{equation*}
	f^{m_j-m_2}\circ h_2=h_2\circ g^{n_j-n_2}
\end{equation*}
on $U'_j$. 
Since $h_1=h_2$ on $D$ we have
\begin{equation*}
	f^{m_j-m_2}\circ h_1=h_1\circ g^{n_j-n_2}
\end{equation*}
on $U'_j$. 
\par  Since $x$ is not a preperiodic point, we have $g^{n_1}(x)\neq g^{n_2}(x)$.  By (\ref{iii}) we know that the diameter of $U_j$ and $U'_j$ shrink to  $0$ when $j\to +\infty$.  Hence we may choose $p>2$ such that $\overline{U_p}\cap \overline{U'_p}=\emptyset$ and $\overline{U_p}\cup \overline{U'_p}\subset D$.    The two maps $g^{n_p-n_1}:U_p\to D$ and $g^{n_q-n_2}:U'_p\to D$ are both biholomorphic.   Let $$W_1:=g_{n_p-n_1}(U'_p)\subset U_p,$$where $g_{n_p-n_1}$ is the inverse map of $g^{n_p-n_1}:U_p\to D$, and let $$W_2:=g_{n_p-n_2}(U_p)\subset U'_p,$$ where $g_{n_q-n_2}$ is the inverse map of $g^{n_p-n_2}:U'_p\to D$.
\par    Let $W:=W_1\cup W_2\subset D$, set $a:=2m_p-m_1-m_2$ and $b:=2n_p-n_1-n_2$, then we have
\begin{equation*}
	f^{a}\circ h_1=h_1\circ g^{b}
\end{equation*}
on $W$.  Moreover $g^{b}:W\to D$ has two repelling fixed point, and is non-injective.
\par 
By Lemma  \ref{inou} there exists an irreducible algebraic curve $\Gamma\subset \P^1(\C)\times \P^1(\C)$ which is invariant under $(f^a,g^b)$ and contains the graph of $h_1$, i.e. the set $\left\{(h_1(w),w)\in \P^1(\C)\times \P^1(\C):w\in D\right\}.$ 
\medskip
\par It remains to show that there exists an irreducible algebraic curve $Z$ containing the graph of $h$ that is  preperiodic under $(f^a,g^b)$. We need the following lemma:
\begin{lem}
	Let $Z\subset \P^1(\C)\times \P^1(\C)$ be an irreducible algebraic curve. Assume there exist $m\geq 0$ and $n\geq 0$ such that $(f^m,g^n)(Z)=\Gamma$. Then $Z$ is preperiodic under $(f^a,g^b)$. 
\end{lem}
\begin{proof}
	Passing to an iteration of $(f^a, g^b)$, we can further assume that $\min\left\{a,b\right\}>\max\left\{n, m\right\}$. For every $l\geq 1$, since $\Gamma$ is invariant under $(f^a,g^b)$, we have 
	$$ (f^{al},g^{bl})(Z)=(f^{a(l-1)+a-m}, g^{b(l-1)+b-n})(\Gamma)=(f^{a-m},g^{b-n})(\Gamma).$$
	\par This means that the $(f^a,g^b)$-forward images of $Z$  are all contained in the curve $(f^{a-m},g^{b-n})(\Gamma)$, which implies that $Z$ is preperiodic under $(f^a,g^b)$.
\end{proof}
\par Let $Z':=\left\{ (h(w),w)\subset \P^1(\C)\times \P^1(\C):w\in U\right\}$ be the graph of $h$. Then we have  $$(f^{m_1}, g^{n_1})(Z')\subset \Gamma.$$
Let $Z$ be the irreducible component of $(f^{m_1}, g^{n_1})^{-1}(\Gamma)$ containing $Z'$. Since $\Gamma$ is irreducible we have 
$$(f^{m_1}, g^{n_1})(Z)=\Gamma.$$
 By the above lemma, $Z$ is preperiodic under $(f^a,g^b)$ and contains the graph of $h$.
This finishes the proof. 
\end{proof}

\subsection{Existence of bi-conical points via ergodic theory}\label{subsectionbiconviaergodic}
We  recall the asymptotic density of subsets of $\Z_{>0}.$
\begin{defi}
	Let $A$ be a subset of positive integers. The {\em asymptotic lower/upper density} of $A$ is defined by 
	\begin{equation*}
		\underline{d}(A):=\liminf_{n\to \infty} |A\cap[0,n-1]|/n,
	\end{equation*}
	and 
	\begin{equation*}
		\overline{d}(A):=\limsup_{n\to \infty} |A\cap[0,n-1]|/n.
	\end{equation*}
	If $\underline{d}(A)=\overline{d}(A)$, we set $d(A):=\underline{d}(A)=\overline{d}(A)$ and call it the  {\em asymptotic density} of $A$.
\end{defi}

\par The proof of the  following lemma was implicitly contained in \cite[Lemma 1]{levin1997two}. For completeness we give a proof here.
\begin{lem}[Levin-Przytycki]\label{lp2}
Let $g$ be an endomorphism on $\P^1(\C)$ of degree at least $2$. Let $\nu$ be a $g$-invariant ergodic probability measure with positive Lyapunov exponent.  Then for every $\epsilon>0$  there is $r>0$ such that for $\nu$-a.e. $x\in \P^1(\C)$, there exists a subset $A=A_x$ of positive integers such that $d(A)> 1-\epsilon$, and  every $n\in A$ satisfies that if we denote by $W$  the connected component of $g^{-n}(B(g^n(x),r))$ containing $x$, then the map $g^n:W\to B(g^n(x),r)$ is biholomorphic and has bounded distortion,
	\begin{equation*}
		1/2<|dg^n(x)/dg^n(y)|<2,
		\end{equation*}
	for all $y \in W$.
\end{lem}

\begin{proof}
\par Let $J:=\Supp \nu$. Consider the Rohlin's natural extension $(\tilde{J},\tilde{g},\tilde{\nu})$ of $(J,g,\nu)$. For $n\in \Z$ denote $\pi_n:\tilde{J}\to J$ the projection to the $n$-th coordinate. Then for $\tilde{\nu}$-a.e. $\tilde{x}\in \tilde{J}$ there exists $r(\tilde{x})>0$ such that for every positive integer $n$, if $W$ is the connected component of $g^{-n}(B(x,r(\tilde{x})))$ containing $\pi_{-n}(\tilde{x})$ where $x:=\pi_0(\tilde{x})$, then the map $g^n:W\to B(x,r(\tilde{x}))$ is biholomorphic and  has bounded distortion, 
	\begin{equation*}
	1/2<|dg^n(x)/dg^n(y)|<2,
\end{equation*}
	for all $y \in W$. Moreover $r(\tilde{x})$ is a measurable function of $\tilde{x}$. See Przytycki-Urbanski \cite[Theorem 11.2.3]{przytycki2010conformal}. 
	\par For fixed $\epsilon>0$, there is $r>0$ such that $\tilde{\nu} (\tilde{F})>1-\epsilon$ where $\tilde{F}:=\{\tilde{x}\in \tilde{J} :  r(\tilde{x})>r\}$. 
By Birkhoff ergodic theorem, there is a subset $\tilde{G}$ with  $\tilde{\nu}(\tilde{G})=1$ such that for every $\tilde{x}\in \tilde{G}$, we have $d(A(\tilde{x}))=\tilde{\nu}(\tilde{F})>1-\epsilon,$
where $A({\tilde{x}}):= \left\{ n\geq 1: \tilde{g}^n(\tilde{x})\in \tilde{F}\right\}$.
Then $\nu(\pi_0(\tilde{G}))=1.$ For $x\in \pi_0(\tilde{G})$, pick $\tilde{x}\in \pi_0^{-1}(x)\cap \tilde{G}$ and set $A_x:=A_{\tilde{x}}$. Then $d(A_x)>1-\epsilon$ 
and $(x,A_x,r)$ has the bounded distortion property we need.
This concludes our proof.
\end{proof}

\medskip
\par  Let $g$ be an endomorphism on $\P^1(\C)$ of degree at least $2$. We let $L$ be a constant strictly larger than the maximum of $|dg|$ on $\P^1(\C)$ w.r.t. the spherical metric. Let $x\in \P^1(\C)$ satisfying $a_n\geq n\log\la-Q$ , where $a_n:=\log|dg^n(x)|$, $\la>1$ and $Q>0$. We define a function $\alpha_{g,x}$ as follows.
\begin{defi}\label{function2}
The function $\alpha_{g,x}:\Z_{\geq 0}\to \Z_{\geq 0}$ is defined as follows: for each $m\in \Z_{\geq 0}$, $\alpha_{g,x}(m)$ is the minimal $n\in \Z_{\geq 0}$ such that $a_n\geq m\log L$. 
\end{defi}
\medskip
\begin{lem}\label{function1}
The function $\alpha_{g,x}$ has the following properties:
\begin{points}
	\item it is strictly increasing, i.e.  $\alpha_{g,x}(m+1)>\alpha_{g,x}(m)$ for every $m\in \Z_{\geq 0}$;
	\item $m\leq \alpha_{g,x}(m)\leq \lceil(m\log L+Q)/\log \la\rceil$.
\end{points} 
\end{lem}
\begin{proof}
We first prove (i). By the definition of $\alpha_{g,x}$, we clearly have $\alpha_{g,x}(m+1)\geq \alpha_{g,x}(m)$. 
Assume by contradiction that $\alpha_{g,x}(m+1)= \alpha_{g,x}(m)=n$ for some $n\in \Z_{\geq 1}$.  Then $a_n\geq (m+1)\log L$, hence $$a_{n-1}=a_n-\log |dg(g^{n-1}(x))|\geq m\log L,$$contradicts to our assumption that $\alpha_{g,x}(m)=n$. 
\par Next we prove (ii).  Since  $a_n\geq n\log\la-Q$ holds for every $n\in \Z_{\geq 0}$, we get $ a_{\lceil(m\log L+Q)/\log \la\rceil}\geq m\log L$. The definition of $L$ implies that $a_m\leq m\log L$. This concludes the proof.
\end{proof}
\medskip
As a direct corollary of Lemma \ref{function1}  we have:
\begin{cor}\label{function3}
Let $A\subset \Z_{\geq 0}$ be a subset such that $d(A)=\delta$, where $0\leq \delta\leq 1$. Then we have:
\begin{equation*}
(\log\la/\log L)\delta \leq  \underline{d}(\alpha_{g,x}(A))\leq \overline{d}(\alpha_{g,x}(A))\leq \delta.
\end{equation*}

\end{cor}

\medskip
Let $U$ be a connected open subset of $\P^1(\C)$. Let $K>1$, we say a  homeomorphism $h:U\to h(U)\subseteq \P^1(\C)$  is {\em $K$-biLipschitz} if for every distinct $x,y\in U$, $K^{-1}(x,y)\leq d(h(x),h(y))\leq Kd(x,y)$, where $d(\cdot,\cdot)$ is the distance function on $\P^1(\C)$. We say that $h$ has {\em biLipschitz} if it is $K$-biLipschitz for some $K>1$.
If $h$ is $K$-biLipschitz, then for every disk $B(x,r)\subset U$, we have $B(h(x),r/K)\subset h(D)\subset B(h(x),Kr)$ for some $r>0$. 

\begin{lem}\label{key}
	Let $f,g$ be two endomorphisms on $\P^1(\C)$ of degree at least $2$.  Let $\mu$ (resp. $\nu$) be an invariant ergodic probability measure with positive Lyapunov exponent of $f$ (resp. $g$). Let $\chi_\mu$ (resp. $\chi_\nu$) be the Lyapunov exponent of $\mu$ (resp. $\nu$). Let $U\subset \P^1(\C)$ be a connected open subset  such that $U\cap \Supp \nu\neq \emptyset$. Let $L$ be a constant larger than the maximum of $|df|$ and $|dg|$ on $\P^1(\C)$.  Let $h:U\to h(U)\subseteq \P^1(\C)$ be a $K$-biLipschitz homeomorphism such that 
	$h_\ast (\nu)$ is equivalent to $\mu$ on $h(U)$. 
	\par Then for $\nu$-a.e. point $x$, $x$  is not $g-$preperiodic and for sufficiently small $\epsilon>0$, there exist positive constants  $r$, $R$, $C$,   a subset $G\subset \Z_{\geq 0}$ with $$\underline{d}(G)\geq \frac{\chi_{\mu}}{\log L}\left(1-\frac{\log L}{\chi_{\nu}}\epsilon\right)-\epsilon>0$$ and a function $\theta:G\to \Z_{\geq 0}$ such that:
	\begin{points}
	
		\item $\lim\limits_{n\to\infty} \frac{1}{n}\log |dg^n(x)|\to \chi_\nu$;
		\item for every $n\in G$, if $W$ is the connected component of the set $g^{-n}(B(g^n(x),r))$ containing $x$, then the map $$g^{n}:W\to B(g^n(x),r)$$  is biholomorphic and has bounded distortion, 
		\begin{equation*}
			1/2<|dg^{n}(x)/dg^{n}(y)|<2,
		\end{equation*}
		for all $y \in W$;
		\item for every $n\in G$, the map $$h_n:=f^{\theta(n)}\circ h\circ g_{n}:B(g^n(x),r)\to \P^1(\C)$$  is injective  and satisfies 
		\begin{equation*}
		B\left(f^{\theta(n)}(h(x)),\frac{R}{16K^2L^2}\right)\subset f^{\theta(n)}(h(W))\subset B(f^{\theta(n)}(h(x)),R)
		\end{equation*}
		where $g_{n}$ is the inverse map of $g^{n}:W\to B(g^n(x),r)$ in (ii);
		\item $\theta$ is  strictly increasing and satisfies 
		\begin{equation*}
		(\chi_\nu /\log L)n-C\leq \theta(n)\leq (\log L/\chi_\mu)n+C
		\end{equation*}
	for every $n\in G$.
	\end{points}
\end{lem}
\begin{proof}
 
Let $F$ be the subset  of $\Supp \nu$ with $\nu(F)=1$ constructed in Lemma \ref{lp2} for the map $g$,  $\epsilon<\chi_\nu/\log L$ and the radius $r>0$.
For $x\in F$, the subset of $\Z_{\geq 0}$ associated with $x$ is denoted by $A_x$. Let $H:=F\cap U$, then $\nu(H)>0$. Since $h_\ast (\nu)$ is equivalent to $\mu$ on $h(U)$, $\mu(h(H))>0.$
Let $E$ be the subset $\Supp \mu$ with $\mu(E)=1$  constructed in Lemma \ref{lp2} for the map $f$, $\epsilon$ and radius $R>0$.  The subset of $\Z_{\geq 0}$ associated with $y$ is denoted by $B_y$. We have $E\cap h(H)\neq \emptyset$.  By Birkhoff ergodic theorem We can choose a point $x\in h^{-1}(E\cap h(H))$ such that $x$ satisfies (i).   Since $\nu$ is non-atomic, we can further choose $x$ not to  be $g-$preperiodic. Set $A:=A_x$ and  $B:=B_{h(x)}$. 
\medskip
\par Recall the definition of the function $\alpha_{g,x}$ in Definition \ref{function2} w.r.t. the constant $L$. We set $\alpha:=\alpha_{g,x}$ and $\alpha':=\alpha_{f,h(x)}$.  We set $$G:= A\cap (\alpha\circ\alpha'^{-1}(B)).$$ Then by $\lim_{n\to\infty} \frac{1}{n}\log |dg^n(x)|\to \chi_\nu$ and  Corollary \ref{function3}, we have $$\underline{d}(G)\geq
\frac{\chi_{\mu}}{\log L}\left(1-\frac{\log L}{\chi_{\nu}}\epsilon\right)-\epsilon>0.$$
By our construction the condition (ii)  automatically holds. 
\medskip
\par It remains to cunstruct the function $\theta$, constant $C$  and verify they satisfy  (iii) and (iv).  For every $n\geq 1$ we set $a_n:=\log |dg^n(x)|$ and $b_n:=\log |df^n(h(x))|$.  
By (ii), for every $n\in A$,  we have 
\begin{equation*}
B(x,re^{-a_n}/2) \subset W\subset B(x,2re^{-a_n}),
\end{equation*}
where $W$ is the connected component of $g^{-n}(B(g^n(x),r))$ containing $x$.  Since $h$ is $K$-biLipschitz we have
\begin{equation*}
	B(h(x),e^{-a_n}r/2K) \subset h(W)\subset B(h(x),2Ke^{-a_n}r).
\end{equation*}
\par Shrinking the constants $r$ and $R$ if necessary we may assume that $2Kr L<R<4Kr L.$
Then 
we have
\begin{equation}\label{2.1}
	B\left(h(x),\frac{Re^{-a_n}}{16K^2 L}\right) \subset h(W)\subset B\left(h(x),\frac{Re^{-a_n}}{2 L}\right).
\end{equation}
\par We define the function $\theta$ as $\theta(n):=\alpha'\circ\alpha^{-1}(n)$.  Since $\alpha^{-1}$ and $\alpha'$ are strictly increasing, $\theta$ is strictly increasing. For every $n\in G$ we have $$\alpha^{-1}(n)\log L\leq a_n<(\alpha^{-1}(n)+1) \log L,$$
hence 
$$a_n/\log L-1<\alpha^{-1}(n)\leq a_n/\log L.$$
Apply 
$$\alpha'^{-1}(k)\log L\leq b_k<(\alpha'^{-1}(k)+1) \log L$$ to $k:=\theta(n)$ we have  
$$\alpha^{-1}(n)\log L\leq b_{\theta(n)}<(\alpha^{-1}(n)+1) \log L. $$
Hence 
\begin{equation}\label{2.2}
a_n-\log L<b_{\theta(n)}<a_n+\log L.	
\end{equation}
\par By Lemma \ref{lp2}, (\ref{2.1})  and (\ref{2.2}) we have 
\begin{equation*}
B(f^{\theta(n)}(h(x)),R/(16K^2L^2))\subset f^{\theta(n)}(h(W))\subset B(f^{\theta(n)}(h(x)),R),
\end{equation*}
hence (iii) holds.
\medskip
\par Finally noticing $L$ is larger than the maximum of $|df|$ and $|dg|$, thus the second part of (iv) holds by  Lemma \ref{function1} (ii). The proof is finished. \end{proof}
\subsection{Holomorphic rigidities}
\proof[Proof of Theorem \ref{measurerigidity}]
 If one of $f,g$ is exceptional, then $\sJ_f$ and $\sJ_g$ are smooth. 
Then by \cite[Theorem 1]{zdunik1990parabolic}, both of them are exceptional. 
So both of them are non-exceptional. 
When $\sJ_g$ (hence $\sJ_f$) is smooth, we may ask $\mu:=\mu_f$ and $\nu:=\mu_g$.

After shrinking $U$, we may assume that $h$ is biLipschitz. Then
Theorem \ref{measurerigidity} is a simple consequence of Lemma \ref{conical} and Lemma \ref{key}.
\endproof
\medskip

The following definition was introduced by Sullivan \cite{sullivan1986quasiconformal}.
\begin{defi}\label{defcer}
	Let $f$ be an endomorphism on $\P^1(\C)$. An compact set $K\subset \P^1(\C)$ is called a CER of $f$ if
	\begin{points}
		\item There exists $m\geq 1$ and a neighborhood $V$ of $K$ such that $f^m(K)=K$ and $K=\cap_{n\geq 0}f^{-mn}(V)$.
		\item $f^m:K\to K$ is expanding.
		\item $f^m:K\to K$ is topologically exact, i.e. for every open set $U\subset K$ there exists $n\geq 0$ such that $f^{mn}(U)=K$.
	\end{points} 
\end{defi}
\proof[Proof of Theorem \ref{conicalrigidity}]
 If one of $f,g$ is exceptional, then $\sJ_f$ and $\sJ_g$ are smooth. 
Then by \cite[Theorem 1]{zdunik1990parabolic}, both of them are exceptional. 
So both of them are non-exceptional.

Let $K$ be a CER of $f$ such that $K\subset h(U)$ which is not a periodic orbit. Such $K$ always exists, see for example \cite[Example 7.4]{ji2022homoclinic}. It is well known that $K$ has positive Hausdorff dimension, see Przytycki-Urbanski \cite[Corollary 9.1.7]{przytycki2010conformal}.  Hence $h^{-1}(K)$ also has positive Hausdorff dimension.  By our assumption (ii),
there exists a point $x\in U\cap \sJ_g$ such that $h(x)\in K$, $x$ is $g$-conical and is not $g$-preperiodic.  Hence there exists $r>0$ and a sequence of positive integers $n_j\to+\infty$ such that 
$$g^{n_j}:W_j\to B(g^{n_j}(x),r)$$
is injective and having bounded distortion, here $W_j$ is the connected component of  $g^{-n_j}(B(g^{n_j}(x),r))$ containing $x$. 
Shrink $U$ if necessary, we may assume that $h$ is biLipschitz.
By  Lemma \ref{derivative},  there exist $C_1>1$ and $r_j>0$, $r_j\to 0$ such that 
$$B(h(x),r_j/C_1)\subset h(W_j)\subset B(h(x),r_j).$$
\par 
After replacing $f$ by a suitable iterate, we may assume that $f(K)=K.$
Since $f|_K$ is uniformly expanding, we know that $|df^{n}(y)|\geq C\la^n$ for every $n\geq 1$ and $y\in K$, where  $C>0$, $\la>1$ are constants. Moreover $d(K, C(f))>0$. Pick $0<R<d(K, C(f))$. Let $m_j$ be the minimal positivet integer such that  $|df^{m_j}(h(x))|r_j\geq R/(2L)$, where $L$ is the supremum of $|df|$ on $\P^1(\C)$.  Hence we have
$$ R/(2Lr_j)\leq |df^{m_j}(h(x))|< R/(2r_j).$$ 
\par Let $V_j$ be the connected component of $f^{-m_j}(B(f^{m_j}(h(x)),R))$. Then
$$f^{m_j}:V_j\to B(f^{m_j}(h(x)),R)$$ is biholomorphic for every $j\geq 1$. Shrink $R$ if necessary,  by Koebe distortion theorem, we have
	\begin{equation*}
	1-99^{-99}<|df^{m_j}(h(x))/df^{m_j}(y)|<1+99^{-99},
\end{equation*}
for all $y \in V_j$.
\par Let $V'_j$ be the connected component of $f^{-m_j}(B(f^{m_j}(h(x)),R/(100C_1L))$.  Then we have $$h(W_j)\subset B(h(x),r_j)\subset V_j$$
and 
$$V'_j\subset B(h(x),r_j/L)\subset h(W_j).$$
\par Hence the map $h_j:=f^{m_j}\circ h\circ g_{n_j}:B(g^{n_j}(x),r)\to \P^1(\C)$  is injective  and satisfies 
\begin{equation*}
	B(f^{m_j}(h(x)),R/(100C_1L))\subset h_j(B(g^{n_j}(x),r))\subset B(f^{m_j}(h(x)),R),
\end{equation*}
where $g_{n_j}$ is the inverse map of $g^{n_j}:W_j\to B(g^{n_j}(x),r)$.  By Lemma \ref{conical}, the conclusion follows.
\endproof
\medskip
\section{Improve local conjugacies to  algebraic correspondences}\label{4}
\subsection{Extend a local conjugacy on a CER}


\medskip
\begin{thm}\label{extendcer}
	Let $f,g$ be two non-Latt\`es endomorphisms on $\P^1(\C)$ of degree at least $2$. Let $K_f$ be an invariant CER of $f$  which is not a periodic orbit.  Let $U$ be a connected neighborhood of $K_f$.  Let $h:U\to h(U)$ be a biholomorphic map such that $h\circ f=g\circ h$ on $K_f$.  Then there exists an irreducible algebraic curve $\Gamma\subset \P^1(\C)\times \P^1(\C)$ which is periodic under $(f,g)$ and contains the graph of $h|_{U}$.
\end{thm}
\begin{proof}
	Let $D$ be a small disk intersecting $K_f$ such that $D\subset U$. Fix a point $x\in K_f\cap D$. Topological exactness of $f|_{K_f}$ implies the preimages of $x$ under $f|_{K_f}$  are dense in $K_f$. Let $x_n\in K_f$ such that $f^n(x_n)=x$. Let $D_n$ be the connected component of $f^{-n}(D)$ containing $x_n$. Since  $f|_{K_f}$ is expanding, this implies that $\diam D_n\leq C\la^{-n}$ for some $C>0,\la>1$.  Since $K_f$ is not a periodic orbit, for $m$ large enough, we can choose two points $x_m\in K_f, x'_m\in K_f$ such that $f^m(x_m)=x$ and $f^m(x'_m)=x$, moreover $D_m\subset\subset D, D'_m\subset\subset D$, and $\overline{D_m}\cap\overline{D'_m}=\emptyset$, where $D_m$ (resp $D'_m$) is the connected component of $f^{-m}(D)$ containing $x_m$ (resp. $x'_m$). Set $W:=D_m\cup D'_m$. Then $f^m:W\to D$ has two repelling fixed point and is non-injective. Moreover $h\circ f^m=g^m\circ h$ on $K_f\cap W$. Since $h$ is holomorphic and $K_f\cap D$ is non-isolated, we have $h\circ f^m=g^m\circ h$ on $W$. By Theorem \ref{inou}, there exists an irreducible algebraic curve $\Gamma\subset \P^1(\C)\times \P^1(\C)$ which is invariant under $(f^m,g^m)$ and contains the graph of $h|_{D}$. 
	Then $\Gamma$ also contains the graph of $h|_{U}$.  The proof is finished. 
\end{proof}
\medskip
The following theorem will be used in the proof of Theorem \ref{interval}. It is a consequence of the following two results:

(1) Sullivan's rigidity theorem for non-linear CERs \cite{sullivan1986quasiconformal} (see \cite[Theorem 7.7]{ji2022homoclinic} for the precise statement  and see \cite[Section 10.2]{przytycki2010conformal} for a proof).

(2) A characterization theorem  of linear CERs \cite[Theorem 1.1]{ji2022homoclinic}. 
\begin{thm}\label{cerrigidity}
	Let $f$ and $g$ be two edomorphisms on $\P^1(\C)$ of degree at least $2$ such that $f$ is non-exceptional. Let $(f,K_f)$, $(g,K_g)$ be two CERs, $f(K_f)=K_f$, $g(K_g)=K_g$.  Let $h:K_f\to K_g$ be a homeomorphism such that $h\circ f=g\circ h$ on $K_f$. Then the following two conditions are equivalent
	\begin{points}
		
		\item for every periodic point $x\in K_f$ we have $|df^n(x)|=|dg^n(h(x))|$, where $n$ is the period of $x$;
		
		\item there exists a  neighborhood $U$ of $K_f$ and a neighborhood $V$ of $K_g$ such that $h$ extends to a holomorphic or antiholomorphic map $h:U\to V$.
	\end{points}
Moreover, in case that any (hence every) condition of (1), (2) holds, then we have 
\begin{points}
\item[{\rm (iii)}] there is an algebraic curve $\Gamma \subseteq \P^1(\C)\times \P^1(\C)$ whose irreducible components are all periodic such that the graph of $h': K_f\to K_{g'}$ is contained in $\Gamma$, where $(g',h')$ is either $(g,h)$ or $(\overline{g}, \tau_{\conj} \circ h)$.	
\end{points}	
	\end{thm}

%
%

\medskip
\subsection{Improve a local conjugacy on an interval}

\begin{defi}\label{repeller}
	Let $M$ be a smooth manifold and let $X\subset M$ be a compact subset. Let $f$ be a continuous map defined on a neighborhood $U$ of $X$ such that $f(X)\subset X$. Then $X$ is called a {\em repeller} if there exists an open set $U'\subset U$ such that $X=\cap_{n\geq 0}f^{-n}(U')$.
\end{defi}
\medskip
\par We need the following lemma. A proof can be found in Przytycki-Urbanski \cite[Lemma 6.1.2]{przytycki2010conformal}. 
\begin{lemma}\label{openrepeller}
	Let $M$ be a smooth manifold and let $X\subset M$ be a compact subset. Let $f$ be a continuous  map defined on a neighborhood $U$ of $X$ such that $f(X)\subset X$.  Then we have:
	\begin{points}
		\item Assume $f|_{U}$ is an open map. If $X$ is a repeller then $f|_{X}$ is an open map.
		\item Conversely if $f:X\to X$ is distance expanding and $f|_{X}$ is an open map, then $X$ is a repeller. 
	\end{points}
\end{lemma}
\medskip
\begin{rem}
 Let $N$ be a smooth manifold, $M\subset N$ be  a submanifold, and $X\subset M$ be a compact subset.  Let $f:N\to N$ be a continuous map such that $f(M)\subset M$. Then there are two definitions for $X$ being repellers, depending on the ambient space is $M$ or $N$.  In the following we will emphasize that $K$ is seen as a subset of which ambient space. In particular when a rational map $f$ is defined over $\R$, then $\R\cup\left\{\infty\right\} $ is invariant under $f$. A compact subset $K\subset \R\cup\left\{\infty\right\} $ may be a repeller for  $f|_{\R\cup\left\{\infty\right\} }$, but not for $f$.
\end{rem}
\medskip
\proof[Proof of Theorem \ref{interval}]
Let $\tilde{f}:=f|_{N_f}$. Since $\tilde{f}$ has positive topological entropy, by the variational principle,  there exists an invariant ergodic measure $\mu$ of $\tilde{f}$ such that the metric entropy $h_\mu(\tilde{f})>0$.  By  Ruelle's inequality \cite[Theorem 10.1.1]{przytycki2010conformal}, $h_\mu(\tilde{f})\leq \max\left\{0, \chi_\mu(\tilde{f})\right\}$, where $\chi_\mu(\tilde{f})$ is the Lyapunov exponent. This implies $\chi_\mu(\tilde{f})$ is a positive finite number.  By Katok's theory for endomorphisms \cite[Theorem 3]{gelfert2010repellers}, there exists a repeller $K$ of $f|_{N_f}$ (here the ambient space is $N_f$), which is topologically exact and uniformly expanding, moreover the topological entropy of $f|_K$ is positive, which implies that $K$ is not a periodic orbit.

 Passing to an iterate of $f$ and $g$, we assume $K$ is $f$-invariant and there exists a repelling $f$-fixed point $x\in K\cap N_f^{\circ}$.  By assumption (ii), $h(x)\in h(K)$ is a repelling fixed point of $g$.  Let $I$ be a small open interval centered at $h(x)$.  Since $g:h(K)\to h(K)$ is topologically exact,  one can construct the following horseshoe map: there are two open intervals $I_1\subset\subset I$, $I_2\subset \subset I$,  such that $g^m:I_1\to I$ and $g^m:I_2\to I$ are bijective for some $m\geq 1$, moreover there exists $L>1$ such that $|dg^m(y)|>L$ for every $y\in I_1\cup I_2$. We define 
$$K_g:=\left\{y\in I_1\cup I_2: f^{mn}(y)\in I_1\cup I_2\;\text{for every}\;n\geq 0\right\}.$$
Then $K_g$ is a repeller of $g|_{N_g}$ (here the ambient space is $N_g$) which is also topologically exact and uniformly expanding.  Since $h(K)$ is also a repeller of $g|_{N_g}$ (here the ambient space is $N_g$), shrinking the interval $I$ if necessary, we have $K_g\subset h(K)$.  We define $K_f:=h^{-1}(K_g)$. Then $K_f$ is a repeller of $f|_{N_f}$ (here the ambient space is $N_f$) which is also topologically exact and uniformly expanding.  
\par By Lemma \ref{openrepeller} (i), since $f^m$ (resp. $g^m$) is an open map on a neighborhood of $K_f$ (resp. $K_g$), $f^m:K_f\to K_f$ and $g^m:K_g\to K_g$ are open maps.  Again by Lemma \ref{openrepeller} (ii), $K_f$ (resp. $K_g$)  is a repeller of $f$ (resp. $g$), where the ambient space is $\P^1(\C)$.  Hence $K_f$ and $K_g$ are CERs.  By Theorem \ref{cerrigidity}, $h|_{K_f}$ can be extended to a holomorphic or antiholomorphic map $\tilde{h}$ on a neighborhood $U$ of $K_f$.  Since $K_f\subset \R\cup \infty$, we can actually let $\tilde{h}$ be holomorphic.  Thus we have $\tilde{h}\circ f^m=g^m\circ \tilde{h}$ on $K_f$. By Theorem \ref{extendcer}, there exists an irreducible algebraic curve $\Gamma\subset \P^1(\C)\times \P^1(\C)$ which is periodic under $(f,g)$, and contains the graph of $h|_{K_f}$.  We conclude the result since by our construction, the graph of $h|_{K_f}$ is a Cantor set. 
\endproof

\section{Real analytic local rigidity of Julia sets} \label{6}
The aim of this section is to prove the following two theorems, which are the real analytic case of Theorem \ref{thmmeasureversion} and Theorem \ref{conicalrigiditytoget} respectively.
\begin{thm}\label{measurerigidityrealan}
Let $f,g$ be two endomorphisms on $\P^1(\C)$ of degree at least $2$.  
Assume that one of them is non-exceptional.
Let $\mu$ (resp. $\nu$) be a non-atomic invariant ergodic probability measure with positive Lyapunov exponent of $f$ (resp. $g$). 
Let $U\subset \P^1(\C)$ be an open subset such that 
$U\cap C_g$ is connected and $U\cap \Supp \nu\neq \emptyset$.  Let $h:U\to h(U)\subseteq \P^1(\C)$ be a real analytic diffeomorphism preserving the orientation such that   
\begin{points}
\item $h(U\cap \sJ_g)=h(U)\cap \sJ_f$; if $\sJ_f$ is smooth, we assume further that $h_\ast (\mu_g)\propto \mu_f$ on $h(U)$;
\item  $h_\ast (\nu)$ is equivalent to  $\mu$ on $h(U)$. 
\end{points}
Then there exist positive integers $a$ and $b$ and an irreducible algebraic curve $Z$ in $\P^1(\C)\times \P^1(\C)$ such that $Z$ is preperiodic under $(f^a,g^b)$ and contains the graph $\{(h(z),z), z\in U\cap C_f\}$ of $h$. 
Moreover, if $\sJ_f$ or $\sJ_g$ is not contained in any circle, then $h$ is holomorphic. 
\end{thm}

\medskip

\begin{thm}\label{conicalrigidityrealan}
Let $f,g$ be two endomorphisms on $\P^1(\C)$ of degree at least $2$. 
Assume that one of them is non-exceptional .
 Let $U\subset \P^1(\C)$ be an open subset such that 
$U\cap C_f$ is connected.
Let $h:U\to h(U)\subseteq \P^1(\C)$ be a real analytic diffeomorphism preserving the orientation. Assume that
\begin{points}
	\item $h(U\cap \sJ_g)=h(U)\cap \sJ_f$; if $\sJ_f$ is smooth, we assume further that $h_\ast (\mu_g)\propto \mu_f$ on $h(U)$;
	\item the Hausdorff dimension of non-conical points of $g$ is $0$.
\end{points}
Then there exist positive integers $a$ and $b$ and an irreducible algebraic curve $Z$ in $\P^1(\C)\times \P^1(\C)$ such that $Z$ is preperiodic under $(f^a,g^b)$ and contains the graph $\{(h(z),z), z\in U\cap C_f\}$ of $h$. 
Moreover, if $\sJ_f$ or $\sJ_g$ is not contained in any circle, then $h$ is holomorphic. 
\end{thm}

\subsection{Replacement of Montel's theorem}
The proof of the following lemma is given by Donyi Wei.
\begin{lem}\label{lemconenormal}Let $g_n: \D\to \D, n\geq 0$ be a sequence of holomorphic maps. Assume that there is a sequence
$r_n\in (0,1)$ tending to $0$ such that $ g_n(\D)\subseteq \D(0, r_n)$.
Let $h: \D\to \D$ be a $C^1$-diffeomorphism with $h(0)=0$ and $|dh|\leq C$ on $\D$ for some $C>0$
Let $f_n: \D(0, Cr_n)\to \D, n\geq 0$ be a sequence of holomorphic maps.
 Set $h_n:=f_n\circ h\circ g_n, n\geq 0.$
 Then there is a subsequence $n_j, j\geq 0$ such that $h_{n_j}$ converges in $C^1$-topology. 
 Moreover $\lim\limits_{j\to\infty} h_{n_j}$ is real analytic.
 
 Assume further that $dh(0)$ is invertible and $f_n,g_n$ are injective, then any non-constant limit $H_\infty=\lim\limits_{j\to\infty} h_{n_j}$ is a homeomorphism to its image and 
the conformal index of $dH_\infty$ is the same as $dh(0)$ at every point in $\D.$
\end{lem}
\proof
For $r\in (0,+\infty)$, define $[r]: \C\to \C$ the map $z\mapsto rz.$
Then we have 
$$h_n=(f_n\circ [r_n]|_{\D(0,C)})\circ ([r_n^{-1}]\circ h\circ [r_n]|_{\D})\circ ([r_n^{-1}]\circ g_n).$$
It is clear that $\{f_n\circ [r_n]|_{\D(0,C)}, n\geq 0\}$ and $\{[r_n^{-1}]\circ g_n, n\geq 0\}$ are normal families. 
Set $H:=dh(0)$. We view $H$ as a $\R$-linear map from $\C$ to $\C$.
Then $h(z)=H(z)+e(z)$ for some $C^1$-map $e$ with $e(0)=0$ and $de(0)=0.$
There is a continuous function $\epsilon: [0,1]\to [0,C]$ with $\epsilon(0)=0$ such that 
$$ |de(z)|\leq \epsilon(|z|).$$
Then $$[r_n^{-1}]\circ h\circ [r_n]|_{\D}=[r_n^{-1}]\circ H\circ [r_n]|_{\D}+[r_n^{-1}]\circ e\circ [r_n]=H+[r_n^{-1}]\circ e\circ [r_n].$$
We claim that $[r_n^{-1}]\circ e\circ [r_n]\to 0$ in $C^1$-topology as $n\to \infty.$
For every $z\in \D$, we have 
$$|[r_n^{-1}]\circ e\circ [r_n](z)|\leq r_n^{-1}r_n\epsilon(r_n)=\epsilon(r_n)$$
and 
$$|d([r_n^{-1}]\circ e\circ [r_n])(z)|= |de(r_nz)|\leq \epsilon(r_n),$$
which concludes the claim.
Then there is a subsequence $n_j, j\geq 0$ such that $f_{n_i}\circ [r_{n_i}]|_{\D(0,C)}$ 
and $[r_{n_i}^{-1}]\circ g_{n_i}$ converges to holomorphic maps $F$ and $G$ respectively. 
Then $\lim\limits_{j\to\infty} h_{n_j}=F\circ H\circ G$ is real analytic.

\medskip

Assume further that $dh(0)$ is invertible  and $f_n,g_n$ are injective. Assume that $h_{n_j}, j\geq 0$ converges to a non-constant map $H_\infty$.
After taking subsequence, we may assume that $f_{n_i}\circ [r_{n_i}]|_{\D(0,C)}$ 
and $[r_{n_i}^{-1}]\circ g_{n_i}$ converges to holomorphic maps $F$ and $G$ respectively. By Hurwitz's theorem,
$F$ (resp. $G$) is either constant or injective. 
Since $H_\infty=F\circ H\circ G$ is non-constant, both $F$ and $G$ are injective.
Since $H$ is a homeomorphism, then $H_\infty=F\circ H\circ G$ is a homeomorphism to its image. Since $F$ and $G$ are holomorphic,  the conformal index of $dH_\infty$ is the same as $dh(0)$ at every point in $\D.$
\endproof

\subsection{Real analytic rigidity of horseshoe}
\begin{thm}\label{thmrealanlyhos}
Let $f,g$ be two endomorphisms on $\P^1(\C)$ of degree at least $2$. Assume that one of $f,g$ is non-exceptional.   
Let $D$ be a simply connected open subset of $\P^1(\C)$ and $W_1, W_2$ be disjoint simply connected open subsets of $D$ compactly contained in $D$
such that $g:W_i\to D$ is biholomorphic, $i=1,2$.  Let $h: D\to h(D)\subseteq \P^1(\C)$ be a real analytic diffeomorphism to its image such that
$h\circ g|_{W_1\cup W_2}=f\circ h|_{W_1\cup W_2}.$ Then $h$ is conformal. 

Moreover,  there exists an irreducible algebraic curve $\Gamma\subset \P^1(\C)\times \P^1(\C)$ which is invariant under $(f,g')$ and contains the graph of $h'$
where $(g',h')$ equals to $(g,h)$ or $(\overline{g}, \sigma_{\conj}\circ h)$ depending on whether $h$ is holomorphic or anti-holomorphic.

\end{thm}

\proof
If $h$ is conformal,  the last statement follows from Theorem \ref{inou}.
We only need to we show that $h$ is conformal.

We may assume that $g$ is non-exceptional.   
Set $W:=W_1\cup W_2.$ Set $K_g:=\{x\in W|\,\, g^{n}(x)\in W \text{ for every } n\geq 0\}$ and $K_f:=h(K_g)$.
Then $(K_g, g^b)$ and $(K_f,f^a)$ are CERs.
By (iii) of Theorem \ref{cerrigidity}, $f$ is non-exceptional.   
If $K_g$ is not contained in any proper real analytic closed subset of $D$, by (iii) of Theorem \ref{cerrigidity}, $h$ is conformal. 
Then we may assume that $K_g$ is contained in a proper real analytic closed subset of $D.$
By  (ii) of Theorem \ref{cerrigidity}, there is a conformal map $h'$ on some neighborhood  $D'$ of $K_g$ such that $h'|_{K_g}=h|_{K_g}.$
Then for every connected component $D''$ of $g^{-1}(D')\cap D'$, if $D''\cap K_g\neq\emptyset$, we have $h'|_{D''}\circ g=f\circ h'|_{D''}.$


There is a connected open subset $U$ in $W\cap D'$ such that $K_g\cap U\neq \emptyset$ and  is contained in a closed smooth connected curve $\gamma$ in $U$.
Let $p$ be any periodic point in $U\cap K_g$. Then there exists $m\geq 1$ such that $g^m(p)=p$ and $dg^m(p)\in (1,+\infty)$. We fix this $m$.

There is a connected open neighborhood $U_1$ of $p$ such that $g^i(U_1)\subseteq W\cap D'$ for every $i=0,\dots, m$ and the map $g^m: U_1\to g^m(U_1)$ conjugates to a linear map $z\in \la^{-1}\D\mapsto \la z\in \D$ where $\la=d g^m(p)$ has norm $>1.$  We may take $U_1$ sufficiently small such that $\overline{\cup_{i=1}^mg^i(U_1)}\neq K_g.$
We have $h\circ g^m|_{U_1}=f^m\circ h|_{U_1}$ and $h'\circ g^m|_{U_1}=f^m\circ h'|_{U_1}$. 
After replacing $f,h, h'$ by $\overline{f}, \sigma_{\conj}\circ h, \sigma_{\conj}\circ h'$, we may assume that $h'$ is holomorphic on $g^m|_{U_1}.$
Since $\la\in \R$, we get 
$df^m(h(p))=\la.$
Fix the coordinate $z$ on $g^m(U_1)$ such that $g^m(U_1)=\{|z|<1\}=\D$ and $g^m|_{U_1}: U_1\to g^m(U_1)$ is $z\mapsto \la z.$
Choose a suitable coordinate on $h'(U_1)$, we may ask $h'=\id$. Then we have $h\circ (\la z)=\la \circ h(z)$ in this coordinate.  Since $\la\in (1,+\infty)$ and $h$ is real analytic, $h$ is $\R$-linear in this coordinate.  Then $K_g\cap g^m(U_1)\subseteq \Fix(h).$ Then $\Fix(h)$ is a line or $h=\id.$ 
If $h=\id$, then we are done, so we may assume that $\Fix(h)$ is a line.
After replacing $z$ by a rotation, we may assume that $K_g\cap g^m(U_1) \subseteq \Fix(h)=g^m(U_1)\cap \R.$


Pick a sequence $p_i, i\geq 0$ such that $p_0=p$, $p_i\in K_g$, $p_{i-1}=g^m(p_i), i\geq 1$ and $p_i\to p$ as $i\to \infty.$ There is $l_0\geq 1$ such that 
$p_{i}\in U_1$ for every $i\geq l_0$.   For every $r\in (0,1)$, let $V_{i}(r)$ be the connected component of $g^{-im}(\D(0,r))$ containing $p_i.$
For $r$ sufficiently small, we have $g^i(V_{l_0}(r))\subseteq D'\cap W$ for every $i=0,\dots,l_0m$ and 
$g^{l_0m}|_{V_{l_0}(r)}: V_{l_0}(r)\to \D(0,r)$ is an isomorphism. 
Then for $l\geq l_0$ sufficiently large, we have $V_{l}(r)\subset\subset \D(0,r)$.
Since $g^{lm}|_{V_{l}(r)}: V_{l}(r)\to \D(0,r)$ is an isomorphism, there is a unique $g^{lm}$-fixed point $q\in V_{l}(r)$.
Set $\mu:=d g^{lm}(q).$ We have $\mu\in (1,+\infty).$ Since $q\in K_g\cap g^m(U_1)$, $q\in \R$ in our coordinate. Let $w$ be the coordinate $w=z-q$.
Since $q\in \Fix(h)$ and $h$ is linear in the coordinate $z$, $h$ is linear in the coordinate $w.$
There is a local holomorphic map $\phi$ at $q$, taking form $\phi(w)=w+O(w^2)$ such that $\phi\circ g^{lm}=\mu\phi$.
Since $g^{lm}\circ h=h\circ g^{lm}$ locally at $q$, we get 
$$\mu(\phi\circ h\circ \phi^{-1})(w)=(\phi\circ h\circ \phi^{-1})(\mu w).$$
The same argument as in the above paragraph shows that $\phi\circ h\circ \phi^{-1}$ is $\R-$linear in $w.$
Hence in coordinate $w$, we get $\phi\circ h\circ \phi^{-1}=d(\phi\circ h\circ \phi^{-1})(0)=h.$
Then we get $\phi\circ h=h\circ \phi.$ Then for $w$ sufficiently closed to $0$, we have
$$d\phi(h(w))\times h= h\times d\phi(w)$$
as $(2\times 2)$-matrix.  We first assume that $d\phi(w)$ is not a constant.
Since $d\phi(w)\in \C$, there is $w\neq 0$ sufficiently closed to $0$ such that 
$\alpha:=\arg d\phi(w)\not\in \pi\Z$.
If $h$ is not conformal, it maps a circle $C$ centered at $0$ to an ellipse $E$ which is not a circle.
Assume that the major axis of $E$ is contained in the line $e^{i\theta}\R$ with $\theta\in [0,\pi).$
Since $d\phi(w)$ and $d\phi(h(w))$ are conformal, the major axis of $h\times d\phi(C)$ is contained in the line $e^{i\theta}\R$ and 
the major axis of $d\phi(h(w))\times h(C)$ is contained in the line $e^{i(\theta+\alpha)}\R$.
Since $i(\theta+\alpha)-i\theta\not\in \pi\Z$, $h\times d\phi(C)\neq d\phi(h(w))\times h(C)$, which is a contradiction.
So $d\phi(w)$ is a constant. In other words, $\phi$ is $\R-$linear in $w$.
Hence in the coordinate $z$, the map $g^{ml}:V_{l}(r)\to \D(0, r)$ is given be that affine map $z\mapsto \mu z+(\mu-1)q$.
Set $W_2':=V_{l}(r)$ and $W_1':=\la^{-l}\D(0,r)$ and $W':=W_1'\cup W_2'$. 
Set $$K':=\{x\in W'|\,\, g^{nlm}(x)\in \D(0,r) \text{ for every } n\geq 0\}.$$ Then $(K', g^{lm})$  is a CER.
Moreover, since in the coordinate $z$, both $g^{lm}|_{W_1'}$ and $g^{lm}|_{W_2'}$ are affine, it is a linear CER \cite[Definition 7.6  (ii)]{ji2022homoclinic}.
By \cite[Theorem 1.1]{ji2022homoclinic}, $g$ is exceptional, which is a contradiction. 
\endproof

\subsection{Real analytic rigidity}
We first prove a real analytic version of Lemma \ref{conical}.
\begin{lemma}\label{conicalrealanaly}
Let $f,g$ be two endomorphisms on $\P^1(\C)$ of degree at least $2$.  Assume that one of them is non-exceptional. 
 Let $U\subset \P^1(\C)$ be an open subset such that 
$U\cap C_f$ is connected.
Let $h:U\to h(U)\subseteq \P^1(\C)$ be a real analytic diffeomorphism preserving the orientation, such that 
 $h(U\cap \sJ_g)=h(U)\cap \sJ_f$,  if $\sJ_f$ is smooth, we assume further that $h_\ast (\mu_g)\propto \mu_f$ on $h(U)$. 
Assume that there exists a  point $x\in U\cap \sJ_g$ which is bi-conical for $(g,h,f)$.
Then there exist positive integers $a$ and $b$ and an irreducible algebraic curve $Z$ in $\P^1(\C)\times \P^1(\C)$ such that $Z$ is preperiodic under $(f^a,g^b)$ and contains the graph $\{(h(z),z), z\in U\cap C_f\}$ of $h$. 
Moreover, if $\sJ_f$ or $\sJ_g$ is not contained in any circle, then $h$ is holomorphic. 
\end{lemma}

\proof[Proof of Lemma \ref{conicalrealanaly}]
 If one of $f,g$ is exceptional, then $\sJ_f$ and $\sJ_g$ are smooth. 
By \cite[Theorem 1]{zdunik1990parabolic}, both of them are exceptional, which contradicts to our assumption.
So both of them are non-exceptional.

\medskip

If $J_g$ is contained in a circle, by Theorem \ref{thmeremen}, $J_f$ is contained in a circle.
We may assume that $C_f=C_g=\R\cup \{\infty\}.$ After shrinking $U$, we may assume that there is a holomorphic injective map $h'$ on $U$ such that $h'=h$ on $U\cap (\R\cup\{\infty\}).$ After replacing $h$ by $h'$, we may assume that $h$ is holomorphic. Then we conclude the proof by Lemma \ref{conical}.

\medskip

Now assume that $\sJ_g$ (hence $\sJ_f$) is not contained in any circle. 
We claim that $h$ is conformal. Once this claim holds, we may conclude the proof by Lemma \ref{conical}.

We now prove the claim.
After shrinking $U$, we may assume that $h$ is biLipschitz. 
Since $x$ is biconical for $(g,h,f)$, it is not $g$-preperiodic and there are positive constants $r,R,K$ and two sequences of positive integers $n_j\to +\infty$, $m_j\to +\infty$, $j\geq 1$ satisfy the conditions (i), (ii) in Definition \ref{defibiconical}.
Same argument as in Lemma \ref{conical} except replacing Levin's theorem \cite[Theorem 1]{Levin1990} (c.f. Theorem \ref{thmlevin}), by Theorem \ref{thmlevinrealan} and 
Montel's theorem by Lemma \ref{lemconenormal},  after taking subsequence, we get $W:=W_1\cup W_2\subset D$ and  $a, b\geq 1$ as the proof of Theorem \ref{thmlevinrealan}, such that 
\begin{equation*}
	f^{a}\circ h_1=h_1\circ g^{b}
\end{equation*}
on $W$.  Moreover $g^{b}:W_i\to D, i=1,2$ are biholomorphic.
By Theorem \ref{thmrealanlyhos}, $h_1$ is conformal. Hence $h$ is conformal. Since $h$ preserves the orientation, $h$ is holomorphic,
which concludes the proof by  Lemma \ref{conical}.
\endproof

\proof[Proof of Theorem \ref{measurerigidityrealan}]
The proof of Theorem \ref{measurerigidityrealan} is the same as 
the proof of Theorem \ref{measurerigidity} except replacing Lemma \ref{conical} by its real analytic version Lemma \ref{conicalrealanaly}.
\endproof

\proof[Proof of Theorem \ref{conicalrigidityrealan}]
The proof is the same as the proof of Theorem \ref{conicalrigidity} except replacing Lemma \ref{conical} by its real analytic version Lemma \ref{conicalrealanaly}.
\endproof

\section{$C^1$ local rigidity of Julia sets}\label{7}
The aim of this section is to prove the following two theorems, which are the $C^1$ case of Theorem \ref{thmmeasureversion} and Theorem \ref{conicalrigiditytoget} respectively.

\begin{thm}\label{measurerigidityrealancone}
Let $f,g$ be two endomorphisms on $\P^1(\C)$ of degree at least $2$.  
Let $\mu$ (resp. $\nu$) be a non-atomic invariant ergodic probability measure with positive Lyapunov exponent of $f$ (resp. $g$). Let $U\subset \P^1(\C)$ be an open subset such that $U\cap \Supp \nu\neq \emptyset$. 
Let $U\subset \P^1(\C)$ be an open subset such that $U\cap \Supp \nu\neq \emptyset$.  Let $h:U\to h(U)\subseteq \P^1(\C)$ be a $C^1$-diffeomorphism such that   
\begin{points}
\item $h(U\cap \sJ_g)=h(U)\cap \sJ_f$; if $\sJ_f$ is smooth, we assume further that $h_\ast (\mu_g)\propto \mu_f$ on $h(U)$;
\item  $h_\ast (\nu)$ is equivalent to  $\mu$ on $h(U)$. 
\end{points}
Then up to change $f$ to $\overline{f}$, there exist positive integers $a$ and $b$ and an irreducible algebraic curve $Z$ in $\P^1(\C)\times \P^1(\C)$ such that $Z$ is periodic under $(f^a,g^b)$.
Moreover, if $\sJ_f$ or $\sJ_g$ is $\P^1(\C)$, then $h$ is conformal.
\end{thm}	

\medskip

\begin{thm}\label{conicalrigiditycone}
Let $f,g$ be two endomorphisms on $\P^1(\C)$ of degree at least $2$. 
Assume that one of them is non-exceptional .
 Let $U\subset \P^1(\C)$ be an open subset such that 
$U\cap C_f$ is connected.
Let $h:U\to h(U)\subseteq \P^1(\C)$ be a $C^1$-diffeomorphism. Assume that
\begin{points}
	\item $h(U\cap \sJ_g)=h(U)\cap \sJ_f$; if $\sJ_f$ is smooth, we assume further that $h_\ast (\mu_g)\propto \mu_f$ on $h(U)$;
	\item the Hausdorff dimension of non-conical points of $g$ is $0$.
\end{points}
Then up to change $f$ to $\overline{f}$, there exist positive integers $a$ and $b$ and an irreducible algebraic curve $Z$ in $\P^1(\C)\times \P^1(\C)$ such that $Z$ is periodic under $(f^a,g^b)$.
Moreover, if $\sJ_f$ or $\sJ_g$ is $\P^1(\C)$, then $h$ is conformal.
\end{thm}

\subsection{Distances between  set of positive density}
Let $G\subseteq \Z_{\geq 0}$ be a subset and $l_1,\dots,l_s\in \Z_{\geq 0}$. Let $G(l_1,...,l_s)$ be the set $n\in G$ such that 
$n+l_1,\dots, n+l_s\in G.$

\begin{lem}\label{lemdisinsetpos}Let $G\subseteq \Z_{\geq 0}$ be a subset with $\overline{d}(G)>0.$ Then for every $l\geq 0$, there is $l_1\geq l$ such that 
$\overline{d}(G(l_1))>0.$
\end{lem}
\proof
We may assume that $l\geq 2.$ For $j=0, \dots, l$,
let $G_{j/l}$ be the set of $n\in G$ with $n=j \mod l.$
Then there is $j$ such that $\overline{d}(G_{j/l})>0.$
After replacing $G$ by $G_{j/l}$, we may assume that for distinct $n_1,n_2\in G$,
$|n_1-n_2|\geq l.$
Since $\overline{d}(G)>0,$ there is $q\in \Z_{\geq l}$ and a strictly increasing sequence 
$N_i, i\geq 0$ such that $$\#(G\cap\{0,\dots, N_i-1\})\geq N_i/q.$$
We may assume that $N_i/100q$ is an integer $r_i.$
Dividing $\{0,..., N_i-1\}$ by $r_i$ segments $I_s=\{100qs,\dots, 100q(s+1)-1\}$, $s=0,\dots, r_i-1.$
It is clear that at least $(0.99N_i)/100q^2$ segments $I_s$ containing at least $2$ elements of $G.$
So $$\overline{d}(\cup_{i=1}^{100q}G(i))>1/(200q^2).$$
Since $G(i)=\emptyset$ for $i=1,\dots, l-1$, there is $i\in \{l,\dots, 100q\}$ such that 
$\overline{d}(G(i))>0$, which concludes the proof.
\endproof

\medskip

Applying Lemma \ref{lemdisinsetpos} inductively, we get the following consequence.
\begin{cor}\label{corsequldis}Let $G\subseteq \Z_{\geq 0}$ be a subset with $\overline{d}(G)>0.$ Then there is a strictly increasing sequence $l_i\geq 1, i\geq 0$ such that for every $s\geq 0,$ $\overline{d}(G(l_1,\dots, l_s))>0.$
\end{cor}

\subsection{$C^1$-rigidity}
\proof[Proof of Theorem \ref{measurerigidityrealancone}]
 If one of $f,g$ is exceptional, then by \cite[Theorem 1]{zdunik1990parabolic}, both of them are exceptional. 
We now assume that both of them are non-exceptional.
When $\sJ_g$ (hence $\sJ_f$) is smooth, we may ask $\mu:=\mu_f$ and $\nu:=\mu_g$.

After shrinking $U$, we may assume that $h$ is $C$-biLipschitz. 
By Lemma \ref{key},  for $\nu|_{U}$-a.e. point $x$, 
it is not $g$-preperiodic and there are positive constants $r,R,K$ and two strictly increasing sequences of positive integers $n_j'\to +\infty$, $m_j'\to +\infty$, $j\geq 1$ satisfy the conditions (i), (ii) in Definition \ref{defibiconical} and $\overline{d}(G)>0$ where $G=\{n_j', j\geq 0\}.$
By Corollary \ref{corsequldis}, there is a strictly increasing sequence $l_i\geq 1, i\geq 1$ such that for every $s\geq 0,$ $\overline{d}(G(l_1,\dots, l_s))>0.$
Define the function $\theta: G\to \Z_{\geq 0}$ by $\theta(n_i')=m_i'.$
Then $\theta$ is strictly increasing.
Pick a sequence $n_j\in \overline{d}(G(l_1,\dots, l_j)), j\geq 1$ and set $m_j:=\theta(n_j)$.
After taking subsequence, we may assume that $g^{n_j}(x)$ converges to a point $p\in \sJ_f.$

\medskip

Define $h_j:=f^{m_j}\circ h\circ g_{n_j}$ as in Lemma \ref{conical} and $D$ be the disc centered at $p$ of radius $r/2$ as in the proof of  Lemma \ref{conical}.
After shrinking $D$ and taking subsequence, by Lemma \ref{lemconenormal},
we may assume that $h_j|_D$ converges to an injective real analytic map $H$.
Each $h_j$ satisfies the following condition: 
$h_j(U\cap \sJ_g)=h_j(U)\cap \sJ_f$,  if $\sJ_f$ is smooth, we assume further that $(h_j)_\ast (\mu_g)\propto \mu_f$ on $h_j(D)$. 
Hence $H$ satisfies the same condition.
\begin{lem}\label{lempbicon}The point $p$ is bi-conical for $(f,H,g).$
\end{lem}
By Lemma \ref{conicalrealanaly}, up to change $f$ to $\overline{f}$, there exist positive integers $a$ and $b$ and an irreducible algebraic curve $Z$ in $\P^1(\C)\times \P^1(\C)$ such that $Z$ is periodic under $(f^a,g^b)$.
 If $\sJ_g$ is not contained in a circle, by Theorem \ref{measurerigidityrealan}, $H$ is conformal. 
 Then $dh(x)$ is conformal by Lemma \ref{lemconenormal}. We conclude the proof.
\endproof

\proof[Proof of Lemma \ref{lempbicon}]
For an open subset $\Omega$ in $\P^1(\C)$ and $y\in \Omega$,
define $$\rho^*(\Omega, y):=\inf\{r\geq 0, \Omega\subseteq B(y,r)\}$$
and $$\rho_*(\Omega, y):=\sup\{r\geq 0, B(y,r)\subseteq \Omega\}.$$

For every $s_1\in G$ and $s_2\in \{0,\dots, s_1\}$, let $W_{s_1/s_2}$ be the connected component of $g^{-(s_1-s_2)}(B(g^{s_1}(x), r))$ containing $g^{s_2}(x).$
Denote by $g_{s_1/s_2}: B(g^{s_1}(x), r)\to W_{s_1/s_2}$ the inverse of the map $g^{s_1-s_2}: W_{s_1/s_2}\to B(g^{s_1}(x), r)$.
For every $u\in G$, define $h'_{u}:=f^{\theta(u)}\circ h\circ g_{u}.$
Set $P:=\{n_i, i\geq 1\}.$
 after shrinking $r$, we may assume that 
$h'_{n_i}$ tends to $H$ uniformly on $D.$
By Koebe distortion theorem,  after shrinking $r$, we may assume that 
for every $s_1\in G$ and $s_2\in \{0,\dots, s_1\}$, the injections $g_{s_1/s_2}$ and $h'_{s_1}$ has good distortion in the following sense:
For every $0<t_1\leq t_2\leq 1$, $$\frac{\rho^*(g_{s_1/s_2}(B(g^{s_1}(x), t_2)), g^{s_2}(x))}{\rho_*(g_{s_1/s_2}(B(g^{s_1}(x), t_1)), g^{s_2}(x))}\leq (1+99^{-999})\frac{t_2}{t_1}$$ 
$$\frac{\rho_*(g_{s_1/s_2}(B(g^{s_1}(x), t_2)), g^{s_2}(x))}{\rho^*(g_{s_1/s_2}(B(g^{s_1}(x), t_1)), g^{s_2}(x))}\geq (1-99^{-999})\frac{t_2}{t_1}$$
$$\frac{\rho^*(g_{s_1}(B(g^{s_1}(x), t_2)), h'_{s_1}(g^{s_1}(x)))}{\rho_*(g_{s_1}(B(g^{s_1}(x), t_1)), h'_{s_1}(g^{s_1}(x)))}\leq (1+99^{-999})\frac{t_2}{t_1}$$
and 
$$\frac{\rho_*(g_{s_1}(B(g^{s_1}(x), t_2)), h'_{s_1}(g^{s_1}(x)))}{\rho^*(g_{s_1}(B(g^{s_1}(x), t_1)), h'_{s_1}(g^{s_1}(x)))}\geq (1-99^{-999})\frac{t_2}{t_1}.$$
\medskip

For every $i\geq 0$, let $W_i$ be the connected component of $g^{-l_i}(B(g^{l_i}(q), r/2))$ containing $q$.
Let $L$ be a constant larger than the maximum of $|df|$ and $|dg|$ on $\P^1(\C)$.
Pick $u_i\in P$ sufficiently large, we have $u_i+l_i\in G$, 
$$d(q,g^{u_i}(x))\leq 99^{-99i}\min\{r,R/K\}L^{-l_i}$$ and for every $z\in D$,
\begin{equation}\label{equhhpe}d(H(g^{u_i}(z)),h'_{u_i}(g^{u_i}z))\leq 99^{-99i}\min\{r,R/K\}L^{-l_i}.
\end{equation}
Since $d(g^{l_i}(q),g^{u_i+l_i}(x))\leq 99^{-99i}\min\{r,R/K\},$ 
we have $$B(g^{l_i+u_i}(x), 0.49r)\subseteq B(g^{l_i}(q), r/2)\subseteq B(g^{l_i+u_i}(x), 0.51r).$$
Then $g_{l_i,q}:=g_{(u_i+l_i)/u_i}|_{B(g^{l_i}(q), r/2)}$ is injective. 
By Lemma \ref{derivative} and the assumption that $g_{(u_i+l_i)/u_i}$ have good distortion, 
$$\diam(g_{(u_i+l_i)/u_i}(B(g^{l_i}(q), r)))\to 0.$$
We may assume that for every $i\geq 1$, 
$g_{(u_i+l_i)/u_i}(B(g^{l_i}(q), r))\subseteq B(q, 0.1r).$
Set $$v_i:=\theta(u_i+l_i)-\theta(u_i).$$

\medskip

Observe that 
$$\rho_*(g_{(u_i+l_i)/u_i}(B(g^{l_i+u_i}(x), 0.49r)), g^{u_i}(x))\geq 0.49rL^{-l_i}.$$
Then we have 
$$\frac{\rho_*(h'_{u_i}(g_{(u_i+l_i)/u_i}(B(g^{l_i+u_i}(x), 0.49r))), g^{u_i}(x))}{\rho^*(h_{u_i}(B(g^{u_i}(x),r), 0.49r))), g^{u_i}(x))}\geq 0.48L^{-l_i}$$
So 
\begin{equation}\label{equlowbabs}\rho_*(h'_{u_i}(g_{(u_i+l_i)/u_i}(B(g^{l_i+u_i}(x), 0.49r))), g^{u_i}(x))\geq 0.48L^{-l_i}R/K.
\end{equation}
By (\ref{equhhpe}),
we have 
\begin{equation}\label{equjhrdown}
\begin{split}
&\rho_*(H(g_{(u_i+l_i)/u_i}(B(g^{l_i+u_i}(x), 0.49r))), h'_{u_i}(g^{u_i}(x)))\\
\geq  & \rho_*(h'_{u_i}(g_{(u_i+l_i)/u_i}(B(g^{l_i+u_i}(x), 0.49r))), h'_{u_i}(g^{u_i}(x)))\\
&-2\times 99^{-99i}\min\{r,R/K\}L^{-l_i}
\end{split}
\end{equation}
and
\begin{equation}\label{equjhrup}
\begin{split}
&\rho^*(H(g_{(u_i+l_i)/u_i}(B(g^{l_i+u_i}(x), 0.51r))), h'_{u_i}(g^{u_i}(x)))\\
\geq  & \rho^*(h'_{u_i}(g_{(u_i+l_i)/u_i}(B(g^{l_i+u_i}(x), 0.51r))), h'_{u_i}(g^{u_i}(x)))\\
&+2\times 99^{-99i}\min\{r,R/K\}L^{-l_i}.
\end{split}
\end{equation}
By (\ref{equlowbabs}) and (\ref{equjhrdown}), we get 
\begin{equation}\label{equjhrdownrel}
\begin{split}
&\rho_*(H(g_{(u_i+l_i)/u_i}(B(g^{l_i+u_i}(x), 0.49r))), h'_{u_i}(g^{u_i}(x)))\\
\geq  & 0.99\rho_*(h'_{u_i}(g_{(u_i+l_i)/u_i}(B(g^{l_i+u_i}(x), 0.49r))), h'_{u_i}(g^{u_i}(x)))\\
\geq  & 0.48\rho_*(h'_{u_i}(g_{(u_i+l_i)/u_i}(B(g^{l_i+u_i}(x), r))), h'_{u_i}(g^{u_i}(x))).
\end{split}
\end{equation}

By (\ref{equlowbabs}) and (\ref{equjhrup}), we get 
\begin{equation}\label{equjhruprel}
\begin{split}
&\rho^*(H(g_{(u_i+l_i)/u_i}(B(g^{l_i+u_i}(x), 0.51r))), h'_{u_i}(g^{u_i}(x)))\\
\leq  & 1.01\rho^*(h'_{u_i}(g_{(u_i+l_i)/u_i}(B(g^{l_i+u_i}(x), 0.51r))), h'_{u_i}(g^{u_i}(x)))\\
\leq  & 0.52\rho^*(h'_{u_i}(g_{(u_i+l_i)/u_i}(B(g^{l_i+u_i}(x), r))), h'_{u_i}(g^{u_i}(x))).
\end{split}
\end{equation}

So we get
\begin{equation}\label{equjhrdownrelend}
\begin{split}
&\rho_*(f^{v_i}(H(g_{(u_i+l_i)/u_i}(B(g^{l_i}(q), r/2)))), f^{v_i}(h'_{u_i}(g^{u_i}(x))))\\
\geq &\rho_*(f^{v_i}(H(g_{(u_i+l_i)/u_i}(B(g^{l_i+u_i}(x), 0.49r)))), f^{v_i}(h'_{u_i}(g^{u_i}(x))))\\
\geq  & 0.47\rho_*(f^{v_i}(h'_{u_i}(g_{(u_i+l_i)/u_i}(B(g^{l_i+u_i}(x), r)))), f^{v_i}(h'_{u_i}(g^{u_i}(x))))\\
\geq &0.47R/K.
\end{split}
\end{equation}
and 
\begin{equation}\label{equjhruprelend}
\begin{split}
&\rho^*(f^{v_i}(H(g_{(u_i+l_i)/u_i}(B(g^{l_i}(q), r/2)))), f^{v_i}(h'_{u_i}(g^{u_i}(x))))\\
\leq &\rho^*(f^{v_i}(H(g_{(u_i+l_i)/u_i}(B(g^{l_i+u_i}(x), 0.51r)))), f^{v_i}(h'_{u_i}(g^{u_i}(x))))\\
\leq  & 0.53\rho^*(f^{v_i}(h'_{u_i}(g_{(u_i+l_i)/u_i}(B(g^{l_i+u_i}(x), r)))), f^{v_i}(h'_{u_i}(g^{u_i}(x))))\\
\leq &0.53R.
\end{split}
\end{equation}
This concludes the proof.
\endproof

\proof[Proof of Theorem \ref{conicalrigiditycone}]
We note that if $\sJ_f$ or $\sJ_g$ is $\P^1(\C)$, then both of them are $\P^1(\C)$.
The proof of Theorem \ref{conicalrigidity} show that there is a bi-conical point $x$ for $(g,h,f)$.
Moreover, if $\sJ_f=\P^1(\C)$, the CER $K$ in the proof of Theorem \ref{conicalrigidity}  con be constructed in any given open subset.
So the bi-conical points for $(g,h,f)$ are dense in $U.$
As in the proof of Theorem \ref{measurerigidityrealancone}, one can construct a real analytic morphism $H: D\to \P^1(\C)$ such that 
$H(D\cap \sJ_g)=H(D)\cap \sJ_f$; if $\sJ_f$ is smooth, we assume further that $H_\ast (\mu_g)\propto \mu_f$ on $H(U)$.
Moreover, $H$ is conformal if and only if $dh(x)$ is conformal.

Apply the argument in the proof of Theorem \ref{conicalrigidity} again, we show that there is a bi-conical point $y$ for $(g,H,f)$.
By Lemma \ref{conicalrealanaly}, Then up to change $f$ to $\overline{f}$,  there exist positive integers $a$ and $b$ and an irreducible algebraic curve $Z$ in $\P^1(\C)\times \P^1(\C)$ such that $Z$ is preperiodic under $(f^a,g^b)$.
Moreover, if $\sJ_f=\P^1(\C)$, then $H$ is conformal. Then $dh(x)$ is conformal. Since such $x$ can be chosen in a dense set in $U$ and $h$ is $C^1$, $h$ is conformal. This concludes the proof.
\endproof


\end{document}